\documentclass[12pt]{amsart}
\usepackage{amsmath}
\usepackage{amsfonts}
\usepackage{latexsym}
\usepackage{amssymb}

\newtheorem{thm}{Theorem}[section]
\newtheorem{cor}[thm]{Corollary}
\newtheorem{lem}[thm]{Lemma}
\newtheorem{prop}[thm]{Proposition}
\newtheorem{ex}[thm]{Example}

\numberwithin{equation}{section}

\newcommand{\ZZ}{{\mathbb Z}}
\newcommand{\RR}{{\mathbb R}}

\newcommand{\Zd} {{\Bbb Z}^d}

\begin{document}
\bibliographystyle{plain}

\title[Polynomial control on stability, inversion and powers of matrices] {Polynomial control on stability, inversion and powers of matrices on simple graphs}

\author{ Chang Eon Shin and Qiyu Sun}

\address{C. E. Shin:  Department of Mathematics, Sogang University, Seoul, 04109,  Korea.}
\email{shinc@sogang.ac.kr }

\address{Q. Sun: Department of Mathematics, University of Central Florida, Orlando, FL 32828, USA.}
\email{qiyu.sun@ucf.edu}


\thanks{  
The authors are  partially supported by the Basic Science Research Program through the National Research Foundation of Korea (NRF) funded by the Ministry of Education, Science and Technology (NRF-2016R1D1A1B03930571), and the National Science Foundation (DMS-1412413).
}




\maketitle
\begin{abstract}  Spatially distributed networks of large size arise in   a variety of science
and engineering problems, such as
wireless sensor networks and
 smart power grids.
Most of their   features  can be described by properties of their state-space matrices
whose entries have indices in the vertex set of a graph.
In this paper, we introduce novel 
algebras of Beurling type that contain matrices on a connected simple graph having polynomial off-diagonal decay, and we show that they are Banach subalgebras of  ${\mathcal B}(\ell^p), 1\le p\le \infty$, the space of all bounded operators on  the space $\ell^p$ of all $p$-summable sequences.  The $\ell^p$-stability of state-space matrices  is an  essential  hypothesis for
the robustness of spatially distributed networks.
In this paper, we establish the  equivalence among $\ell^p$-stabilities of matrices in  Beurling algebras for different exponents $1\le p\le \infty$,       with  quantitative analysis for the lower stability bounds.
Admission of norm-control inversion
plays a crucial role in some engineering practice. In this paper, we prove that matrices in Beurling subalgebras of ${\mathcal B}(\ell^2)$
have norm-controlled inversion and we find a norm-controlled polynomial with close to optimal degree.
Polynomial estimate to powers of matrices  is important for numerical implementation of spatially distributed networks.
In this paper, we apply our results on norm-controlled
inversion to obtain a  polynomial estimate  to powers of matrices in  Beurling algebras. The polynomial estimate
is a noncommutative extension about convolution powers of a complex function and
is applicable to    estimate the probability  of hopping from one agent  to another agent in a stationary Markov chain   on a spatially distributed network.
\end{abstract}

\section{Introduction}

A spatially distributed network  (SDN) contains a large amount of agents
with limited sensing, data processing, and communication capabilities  for information transmission.
It arises in a variety of science and engineering problems  (\cite{aky02,  chong2003, 
 hebner2017, 
  yick2008}).
The  topology  of an SDN  can be  described  
 by a  
 graph
 \begin{equation}\label{g.graph}
 {\mathcal G}:=(V, E)\end{equation}
 of large size,
 where a vertex in $V$
represents  
 an agent
and an  edge  $(\lambda,\lambda')\in E$
between two vertices $\lambda$ and $\lambda'\in V$ means that a  direct communication link exists. In this paper, we always assume that ${\mathcal G}$ is
 connected and simple. Here a simple graph means that it is  an unweighted, undirected graph
containing no graph loops or multiple edges.
Our motivating examples  are  1)  circular graphs ${\mathcal Z}_N:=(\ZZ/N\ZZ, E_N)$ of order $N\ge 3$, where  $(m,n)\in E_N$ means $m-n\in N\ZZ\pm 1$; and 2) lattice graphs
${\mathcal Z}^d:=(\Zd, E^d), d\ge 1$,
where
$(m,n)\in E^d$ implies that $m$ and $n\in \ZZ^d$ have distance one.
For the graph to describe an SDN, the assumption on its connectivity and simpleness can be understood as that
agents in the SDN can communicate across the entire
network, direct communication links between agents are bidirectional,
agents have the same communication specification,  communication components are not used for data transmission within an agent,
and no multiple direct communication channels  exist between agents \cite{chengsun}.

\smallskip
SDNs could give extraordinary 
capabilities especially when creating
a data exchange network requires significant efforts
or when establishing a centralized facility to process and store
all the information is formidable.  
 A 
comprehensive  mathematical analysis of  SDNs
does not appear to exist yet, 
 and
there is a huge research gap between mathematical theory and engineering practice  \cite{bamieh02, chengsun, Dullerud2004, mjieee08, napoli99,   smale07}. This inspires us to consider various properties of
state-space matrices
\begin{equation}\label{matrix.graph} A=\big(a(\lambda, \lambda') \big)_{\lambda, \lambda' \in V}\end{equation}
of SDNs with indices in the vertex set $V$ of  a graph.  This  work 
 is also motivated by the emerging field
of signal processing on graphs, where matrices of the form \eqref{matrix.graph} are used for  linear processing such as filtering, translation, modulation,
dilation and  downsampling \cite{chen15,  pesenson08,  moura14, moura13,  shuman13}.

An abundant family of SDNs is spatially decaying linear systems whose  state-space matrices have off-diagonal
decay.
 Examples of such systems include 
smart power  grids with sparse interconnection topologies,
multi-agent systems with nearest-neighbor coupling structures, and
(wireless) sensor networks for environment monitoring (\cite{aky02, chengsun, chong2003, Dullerud2004,  hebner2017,  mjieee08, yick2008}).
To describe  off-diagonal decay property of matrices of the form \eqref{matrix.graph}, we introduce
Banach algebras  ${\mathcal B}_{r, \alpha}({\mathcal G})$ of   Beurling type for $1\le r\le \infty$ and $\alpha\ge 0$, see \eqref{beurling.def}
and \eqref{bralpha.norm} in Section \ref{beurling.section}.  Matrices $A=(a(\lambda, \lambda') )_{\lambda, \lambda' \in V}$
in ${\mathcal B}_{r, \alpha}({\mathcal G})$ have their entries  dominated by
a positive decreasing function  $h_A$ with  polynomial  decay,
\begin{equation}
|a(\lambda, \lambda')|\le h_A(\rho(\lambda, \lambda')) \ \ {\rm for \ all}\  \lambda, \lambda'\in V,
\end{equation}
where $\rho(\lambda, \lambda')$ is the geodesic distance between vertices $\lambda, \lambda'\in V$.
For the  lattice graph
${\mathcal Z}^d$, Banach algebras ${\mathcal B}_{r, \alpha}({\mathcal Z}^d)$ are introduced by Beurling  in
 \cite{beurling49}
 for $r=d=1$ and $\alpha=0$, Jaffard in \cite{jaffard90} for $r=\infty$  and Sun in \cite{sunca11}  for $1\le r\le \infty$.

\smallskip
Let $\ell^p:=\ell^p(V), 1\le p\le \infty$, be   Banach spaces of $p$-summable sequences on $V$ with standard norm  $\|\cdot\|_{\ell^p}$,
and ${\mathcal B}(\ell^p)$ be the Banach algebra of all bounded operators on $\ell^p$ with norm $\|\cdot\|_{{\mathcal B}(\ell^p)}$.
We say that
 a  matrix $A\in {\mathcal B}(\ell^p)$ has {\em $\ell^p$-stability} if
  there exists a positive constant $A_p$  such that
\begin{equation}\label{stability.def}
\|Ac\|_{p} \ge A_p \|c\|_{p}    \ \ {\rm for \ all} \ c\in {\ell^p}.
\end{equation}
The
optimal lower $\ell^p$-stability
bound of a matrix $A$
is the maximal number $A_p$   for \eqref{stability.def} to hold.
The $\ell^p$-stability  is an  essential  hypothesis for
some matrices arising in time-frequency analysis, sampling theory, wavelet analysis and many other applied mathematical fields  \cite{aldroubisiamreview01,  
christensenbook,  grochenigbook, 
sunsiam06, xiansun14}.
For the robustness  against bounded noises,
 the sensing matrix  arisen in the sampling and reconstruction procedure of signals on a SDN is required in \cite{chengsun} to have  $\ell^\infty$-stability, however there are some difficulties to verify  $\ell^p$-stability  of a matrix  in a distributed  manner  for $p\ne 2$.

  For a finite graph ${\mathcal G}=(V, E)$ and a matrix $A$ with indices in its vertex set $V$, its $\ell^p$-stability and $\ell^q$-stability  are equivalent to each other for any $1\le p, q\le \infty$, and
its optimal lower stability bounds satisfy
\begin{equation}\label{Mstability}
M^{-|1/p-1/q|}\le \frac{A_q}{A_p}\le  M^{|1/p-1/q|},
\end{equation}
where $M=\# V$ is the number of vertices of the graph ${\mathcal G}$. The above estimation on lower stability bounds is unfavorable for matrices of large size but it cannot be improved  if there is no restriction on the matrix $A$.
Let $d$ be the Beurling dimension of the graph ${\mathcal G}$.  Matrices $A\in {\mathcal B}_{r, \alpha}({\mathcal G})$ with  $1\le r\le \infty$ and $\alpha>d(1-1/r)$ are bounded operators on $\ell^p, 1\le p\le \infty$, and there exists a positive constant $C$  such that
 \begin{equation*}  
 \|A\|_{{\mathcal B}(\ell^p)}
\le C \|A\|_{{\mathcal B}_{r,\alpha}} \ {\rm for \ all} \ A\in {\mathcal B}_{r, \alpha} ({\mathcal G}) \ {\rm and}\ 1\le p\le \infty.
 \end{equation*}
For their lower stability bounds, it is proved  that
\begin{equation*}
A_q>0 \ {\rm if \ and \ only \ if} \ A_p>0
\end{equation*}
in \cite{akramjfa09, sunca11, tesserajfa10} for the infinite lattice graph ${\mathcal Z}^d$,
and that
\begin{equation*}
 A_q>0 \ {\rm if } \ A_2>0
\end{equation*}
 in \cite{chengsun} for any  infinite graph ${\mathcal G}$ with finite
  Beurling dimension and
 $r=\infty$, where $1\le p, q\le \infty$.
In Theorem \ref{stability.thm} of this paper, 
 we provide a quantitative version of  $\ell^p$-stability for  different $1\le p\le \infty$ and prove the following result,
\begin{equation}
  D_1
\Big(\frac{\|A\|_{{\mathcal B}_{r, \alpha}}}{A_p}\Big)^{-D_0|1/p-1/q|}\le
\frac{A_q}{A_p}
 \le  D_2
\Big(\frac{\|A\|_{{\mathcal B}_{r, \alpha}}}{A_p}\Big)^{D_0|1/p-1/q|},
\end{equation}
 for any matrix $A\in {\mathcal B}_{r, \alpha}({\mathcal G})$ with  $1\le r\le \infty$ and $\alpha>d(1-1/r)$,
where $D_0, D_1, D_2$ are  absolute constants {\bf independent} of matrices $A\in {\mathcal B}_{r, \alpha}({\mathcal G})$, exponents $1\le p, q\le \infty$  and the size $M$ of the graph ${\mathcal G}$, cf. \eqref{Mstability}.
The  proof of Theorem \ref{stability.thm} depends on an important estimate to the commutator between
a matrix in the Beurling algebra  ${\mathcal B}_{r, \alpha}({\mathcal G})$  and  a truncation operator. Similar estimate has been used by Sj\"ostrand  in \cite{sjostrand} to establish invertibility of
infinite matrices in the Baskakov-Gohberg-Sj\"ostrand class.

\smallskip
A Banach subalgebra ${\mathcal A}$ of ${\mathcal B}$ is said to be  {\em
inverse-closed}   if an element in ${\mathcal A}$, that is invertible in ${\mathcal B}$, is also invertible in ${\mathcal A}$.
The inverse-closed subalgebras have numerous applications in time-frequency analysis, sampling theory, numerical analysis and optimization,
 and it  has been established
for matrices, integral operators, pseudo-differential operators satisfying
various off-diagonal decay conditions.
The reader may refer to  \cite{balan, baskakov90, 
dahlke,
farrell10, grochenig10, grochenig06, 
 grochenigklotz10,
gltams06,  
 jaffard90, kristal15, 
 moteesun,  sjostrand, 
sunca11,
 suntams07}
    and   the  survey papers \cite{grochenig10, Krishtal11, shinsun13} for historical remarks and recent advances.

A quantitative version of inverse-closedness is the {\em norm-controlled inversion}   \cite{grochenig13-1,  grochenig14-1, nikolski99, romero09, stafney}.
Here an inverse-closed Banach subalgebra ${\mathcal A}$   of ${\mathcal B}$ is said to
 admit  norm control in ${\mathcal B}$
if there exists a continuous function $h$ from $[0, \infty)\times [0, \infty)$ to $[0, \infty)$ such that
\begin{equation}  \label{normcontrol.def}
\|A^{-1} \|_{\mathcal A}\le h(\|A\|_{\mathcal A}, \|A^{-1}\|_{\mathcal B})\end{equation}
for all $A\in {\mathcal A}$ with $A^{-1}\in {\mathcal B}$.
Admission of norm-control inversion  plays  a crucial role in \cite{sunaicm14}  to solve  nonlinear sampling problems  termed with instantaneous companding  and local identification of signals with finite rate of innovation.
Norm-controlled inversion was first studied by Stafney in \cite{stafney}, where it is shown that ${\mathcal B}_{1, 0}({\mathcal Z})$
 does not admit a norm-controlled inversion in ${\mathcal B}(\ell^2)$.
 The polynomial norm-control inversion is established in  \cite{grochenig13-1} for matrices in  differential algebras
  and  \cite{grochenig14-1} for matrices in Besov algebras, Bessel algebras, Dales-Davie algebras and Jaffard algebra.
  In Theorem \ref{inverse.thm} of this paper, we show that Banach algebra
  ${\mathcal B}_{r, \alpha}({\mathcal G})$ with  $1\le r\le \infty$ and $\alpha>d(1-1/r)+1$ admit norm-controlled inversion in ${\mathcal B}(\ell^2)$, and there exists an absolute constant $C$ such that
  \begin{equation}\label{inverseestimate}
  \|A^{-1}\|_{{\mathcal B}_{r, \alpha}}  \le C
\|A^{-1}\|_{{\mathcal B}(\ell^2)}
(\|A^{-1}\|_{{\mathcal B}(\ell^2)}\|A\|_{{\mathcal B}_{r, \alpha}})^{\alpha+d/r}
  \end{equation}
  hold for all $A\in {\mathcal B}_{r, \alpha}({\mathcal G})$ with $A^{-1}\in {\mathcal B}(\ell^2)$.  Moreover, the above polynomial norm-control inversion is close to optimal, as shown in
   Example \ref{agamma.example} that the exponent $\alpha+d/r$  in \eqref{inverseestimate} cannot be replaced by $\alpha+d/r-1-\epsilon$
    for any $\epsilon>0$.
    We remark that a weak version of the norm-controlled  estimate \eqref{inverseestimate}, with the exponent $\alpha+d/r$ in \eqref{inverse.cor.eq1} replaced by a larger exponent $2\alpha+2+2/(\alpha-2)$, is established in \cite{grochenig14-1} for the Jaffard algebra ${\mathcal J}_\alpha({\mathcal Z})={\mathcal B}_{\infty, \alpha}({\mathcal Z})$.


\smallskip
Let ${\mathcal A}$ be a Banach subalgebra of ${\mathcal B}$. We say that ${\mathcal A}$   is its
 {\em differential subalgebra of order $\theta\in (0, 1]$} (\cite{blackadarcuntz91, christ88, kissin94, sunaicm14, suncasp05})  if
there exists a positive constant $C$ such that
\begin{equation}\label{differentialnorm.def}
\|AB\|_{\mathcal A}\le C\|A\|_{\mathcal A} \|B\|_{\mathcal A} \Big (\Big(\frac{\|A\|_{\mathcal B}}{\|A\|_{\mathcal A}}\Big)^\theta +
\Big(\frac{\|B\|_{\mathcal B}}{\|B\|_{\mathcal A}}\Big)^\theta
\Big)
\quad {\rm for \ all} \  A, B \in {\mathcal A}.
\end{equation}
The  differential subalgebras  was introduced in \cite{blackadarcuntz91, kissin94} for $\theta=1$ and
\cite{christ88, sunaicm14, suncasp05} for $\theta\in (0, 1)$. 
In  \cite{christ88, grochenigklotz10}, it is shown that  a $C^*$-subalgebras ${\mathcal A}$ of ${\mathcal B}$  with a common unit admits  norm controlled inversion
if ${\mathcal A}$ is also a differential subalgebra.
The reader may refer to
\cite{blackadarcuntz91,  christ88, kissin94, moteesun,  
shinsun13, sunaicm14,  suntang} and references therein
for historical remarks and recent advances in
 operator theory, harmonic analysis, non-commutative geometry, numerical analysis and  optimization.
The differential norm inequality \eqref{differentialnorm.def}
is satisfied by many  Banach subalgebras of ${\mathcal B}(\ell^2)$ \cite{grochenigklotz10,  sunca11, suntams07, suncasp05}.
For $1\le r\le \infty$ and $\alpha>d(1-1/r)$,  it is shown in Proposition \ref{subdiff.prop} that Banach algebras
  ${\mathcal B}_{r, \alpha}({\mathcal G})$
  are differential subalgebra of ${\mathcal B}(\ell^2)$ with order   $\theta_{r, \alpha}=2(\alpha-d+d/r)/(1+2\alpha-2d+2d/r)$.
  Applying  the differential property \eqref{differentialnorm.def}
 repeatedly and using the argument in \cite[Proposition 2.4]{sunaicm14}, we have the following subexponential estimate for  the norms of powers $A^n, n\ge 1$, in ${\mathcal B}_{r, \alpha}({\mathcal G})$,
\begin{equation*} 
\|A^n\|_{{\mathcal B}_{r, \alpha}}\le   \|A\|_{{\mathcal B}(\ell^2)}^n
\Bigg( C\frac{ \|A\|_{{\mathcal B}_{r, \alpha}}}{\|A\|_{{\mathcal B}(\ell^2)}}\Bigg)^{\frac{\theta_{r, \alpha}}{1+\theta_{r, \alpha}} n^{\log_2 (1+\theta_{r, \alpha})}}, \ \ n\ge 1,\end{equation*}
where $C$ is an absolute constant independent of integers $n\ge 1$ and  matrices $A\in {\mathcal B}_{r, \alpha}({\mathcal G})$.
In Theorem \ref{power.thm} of this paper, we refine the above estimate to show that  powers of a matrix in
${\mathcal B}_{r, \alpha}({\mathcal G})$ with $1\le r\le \infty$ and $\alpha>d(1-1/r)+1$ have polynomial growth,
\begin{equation}  \label{power.thm.eq000}
\|A^n\|_{{\mathcal B}_{r, \alpha}}
\le   C     n
\Big(
\frac{ n
\|A\|_{{\mathcal B}_{r, \alpha}}}{
\|A\|_{{\mathcal B}(\ell^2)} } \Big)^{\alpha+d/r}   \|A\|_{{\mathcal B}(\ell^2)}^{n}, \ n\ge 1.\end{equation}
 Moreover, the above estimate is close to optimal, as shown in \eqref{power.cor.eq2} that
the exponent $\alpha+d/r$ in \eqref{power.thm.eq000} cannot be replaced by $\alpha+d/r-1-\epsilon$ for any $\epsilon>0$.
The polynomial norm estimate in \eqref{power.thm.eq000} is a noncommutative extension about convolution powers of a complex function on $\Zd$, cf. \cite{randle14,  randle15, thomee69, thomee65} and
 \eqref{ank.estimate}.
 The power estimate in \eqref {power.thm.eq000} is also applicable to estimate the probability  $ \Pr(X_{n}=\lambda |X_{1}=\lambda') $ of hopping from one agent $\lambda$ to another agent $\lambda'$ in a stationary Markov chain  $X_n, n\ge 1$, on a spatially distributed network, see Corollary \ref{markov.cor}.

\smallskip
The paper is organized as follows. In Section \ref{preliminary.section},  we  recall some  preliminaries on  connected simple graphs  ${\mathcal G}$ and  provide two basic estimates 
about their geometry. 
In Section \ref{beurling.section},
 we introduce novel  algebras   of matrices, ${\mathcal B}_{r, \alpha}({\mathcal G})$ with $1\le r\le \infty$ and $\alpha\ge 0$,
  and we
prove that they are differential and  inverse-closed subalgebra of ${\mathcal B}(\ell^2)$.
In Section \ref{stability.section}, we establish the equivalence among  $\ell^p$-stabilities of matrices in the Beurling algebra ${\mathcal B}_{r, \alpha}({\mathcal G})$  for different exponents $1\le p\le \infty$, and we further show that their lower stability bounds are controlled  by some polynomials.
In Section \ref{inversion.section}, we prove that the Beurling  algebra ${\mathcal B}_{r, \alpha}({\mathcal G})$
 admit norm-controlled inversion and a polynomial can be selected to be the norm-controlled function  $h$ in \eqref{normcontrol.def}.
In Section \ref{power.section}, we consider noncommutative extension of convolution powers and show that
 norms of powers $A^n, n\ge 1$, of a matrix  $A\in {\mathcal B}_{r,\alpha}$ with $1\le r\le \infty$ and $\alpha>d(1-1/r)$ are dominated by a   polynomial.

\smallskip
Notation: $\ZZ_+$ contains all nonnegative integers,   $\lfloor t\rfloor, \lceil t\rceil$ and  $t_+=\max(t, 0)$ of a number $t$
are
the greatest preceding integer, the least succeeding integers and the positive part respectively,
and
for a set $F$
denote its cardinality and characteristic function by $\#F$  and $\chi_F$  respectively.
In this paper, the capital letter $C$ is an absolute constant which is not necessarily the same at each occurrence.

\section{Preliminaries on connected simple graphs}
\label{preliminary.section}

In this section, we  recall some  concepts on  connected simple graphs, and we
establish some estimates  
about their geometry. 

\smallskip
 Let $\mathcal{G}:=(V,E)$ be a connected simple graph.
Denote   by $\rho$
the {\em geodesic distance}  on  ${\mathcal G}$, which is the nonnegative function  on $V\times V$ such that $\rho(\lambda,\lambda)=0$  for  all vertices $\lambda\in V$,  and  $\rho(\lambda,\lambda')$ is the number of edges in a shortest path
connecting distinct vertices $\lambda, \lambda'\in V$ (\cite{chungbook}).
For some real-world applications of SDNs,
communication between  two distinct agents happens by transmitting information through the chain of intermediate agents connecting them using a shortest path, and  the geodesic distance  is widely used to measure the communication cost  to data exchange.

The geodesic distance $\rho$ on ${\mathcal G}$ is a metric on $V$.
For the simple graph
${\mathcal Z}^d $, 
the geodesic distance between $m=(m_1, \ldots, m_d)$ and $n=(n_1, \ldots, n_d)\in \Zd$ is given by
$\rho(m, n)=\sum_{i=1}^d |m_i-n_i|.$

With the  geodesic distance $\rho$ on ${\mathcal G}:=(V,E)$, we denote the closed ball with center $\lambda\in V$ and radius $R$ by
$$ B(\lambda, R):=\{\lambda' \in V: \ \rho(\lambda, \lambda')\le R \}, $$
and  the {\em counting measure} on $V$  
 by $\mu$, where $\mu(F)$ is the number of vertices in $F$ for any $F\subset V$. The counting measure $\mu$ is  said to be a {\it  doubling measure} (\cite{ms79, yyh13}) if there exists a positive constant $D_0({\mathcal G})$ such that
\begin{equation}\label{doubling}
\mu(B(\lambda, 2R))\le D_0({\mathcal G}) \mu(B(\lambda,R)) \  \ \text{for all}\ \lambda \in V \text{ and } R\ge 0.
\end{equation}
The minimal constant $ D_0({\mathcal G})$ in \eqref{doubling}  is known as the {\em doubling constant} of the measure $\mu$.
Under the doubling assumption to the measure $\mu$, the triple $(V, \rho, \mu)$ is   a space of homogeneous type. 
 The reader may refer to \cite{ms79, yyh13} for harmonic analysis on spaces of homogeneous type.

We say that the counting measure $\mu$ on the graph ${\mathcal G}$  has {\it  polynomial growth} if there exist positive constants $D_1({\mathcal G})$ and $d:=d({\mathcal G})$ such that
\begin{equation}\label{polynomialgrowth}
\mu(B(\lambda,R))\le D_1({\mathcal G}) (R+1)^{d} \ \ \text{for all} \ \lambda\in V \text{ and } R\ge 0.
\end{equation}
The minimal constants  $d$ and $D_1({\mathcal G})$  in
 \eqref{polynomialgrowth}
 are called as the {\em Beurling dimension} and {\em density} of the graph ${\mathcal G}$  (\cite{chengsun}).  
For the simple graph ${\mathcal Z}^d$, its  Beurling dimension   is the same as the Euclidean dimension $d$.
  We remark that a simple graph  ${\mathcal G}$  with its counting measure $\mu$ satisfying the doubling condition \eqref{doubling} has finite Beurling dimension,
 \begin{equation*}
\mu(B(\lambda,R))\le D_0({\mathcal G}) (R+1)^{\log_2 D_0({\mathcal G})} \ \ \text{ for all}\  \lambda\in V \text{ and } R\ge 0,
\end{equation*}
 where $D_0({\mathcal G})$ is the doubling constant of the  measure $\mu$.

We say that the  counting measure  $\mu$ on the graph ${\mathcal G}$  is
{\em normal} if  there exist $D_1({\mathcal G})$ and $D_2({\mathcal G})$ such that
\begin{equation}\label{ahlfors.def}
 D_2({\mathcal G}) (R+1)^d \le \mu(B(\lambda, R))\le  D_1({\mathcal G}) (R+1)^{d}
\end{equation}
  for all $\lambda\in V$  and  $0\le R\le {\rm diam}({\mathcal G})$, the diameter of the graph ${\mathcal G}$.
  One may verify that the counting measures on ${\mathcal Z}^d$  and  ${\mathcal Z}_N, N\ge 1$, are normal, and
    a normal measure has the doubling property
\eqref{doubling} and the polynomial growth property  \eqref{polynomialgrowth}.
 The reader may refer to \cite{keith08, ms79, tyson01, yyh13} and references therein
for normal measures.

\smallskip

We conclude this section with a proposition 
on geometry of a connected simple graph with finite Beurling dimension.

\begin{prop} \label{graph.pr}
Let ${\mathcal G}:=(V,E)$ be a connected simple graph with Beurling dimension $d$,
and  $h:=\{h(n)\}_{n=0}^\infty$ be a positive decreasing sequence.
 Then the following statements hold.

\begin{itemize}

\item[{(i)}] For any vertex $\lambda\in V$ and   integer $s\ge 0$,
\begin{equation}\label{sum2}
 \sum_{\rho(\lambda, \lambda')\le s} \rho(\lambda, \lambda')  h(\rho(\lambda, \lambda')) \le (d+1) D_1({\mathcal G}) \sum_{n=0}^s h(n) (n+1)^{d}.
\end{equation}

\item[{(ii)}] If  $\sum_{n=0}^\infty  h(n) (n+1)^{d-1}<\infty$, then
\begin{equation} \label{jaffardpr.pf.eq2}
\sum_{\rho(\lambda, \lambda')\ge s} h(\rho(\lambda,\lambda'))
\le  D_1({\mathcal G})\Big( (s+1)^d h(s)+ d\sum_{n=s+1}^\infty h(n) (n+1)^{d-1}\Big)
\end{equation}
for any vertex $\lambda\in V$ and   integer $s\ge 0$.
\end{itemize} 
\end{prop}

\begin{proof}  
(i).\   
Given  a vertex $\lambda\in V$ and an integer $s\ge 0$, we obtain
 \begin{eqnarray*}  
& \hskip-0.08in  &  \hskip-0.08in  \sum_{\rho(\lambda, \lambda')\le s} \rho(\lambda, \lambda') h(\rho(\lambda,\lambda'))\\
  &  \hskip-0.08in  = &  \hskip-0.08in    \sum_{n=0}^s  n h(n) \big(\mu(B(\lambda,n))- \mu(B(\lambda,n-1))\big)
 \nonumber\\
 & \hskip-0.08in  = &  \hskip-0.08in
 s h(s) \mu(B(\lambda,s))+  \sum_{n=0}^{s-1} \mu(B(\lambda,n)) \big (n h(n)-(n+1) h(n+1)\big)\nonumber\\
 & \hskip-0.08in  \le  &  \hskip-0.08in  D_1({\mathcal G}) \Big(
  s h(s) (s+1)^d+  \sum_{n=0}^{s-1} n (h(n)-h(n+1)) (n+1)^d\Big)\nonumber\\
 & \hskip-0.08in  = & \hskip-0.08in  D_1({\mathcal G})  \sum_{n=1}^{s}  h(n) \big(n (n+1)^d-(n-1) n^d\big),
 \end{eqnarray*}
 where the  inequality follows from the polynomial growth  property \eqref{polynomialgrowth} and the monotonic assumption on the nonnegative sequence $\{h(n)\}_{n=0}^\infty$.
  Hence \eqref{sum2} follows.

(ii). \
Take a vertex $\lambda\in V$ and  an integer $s\ge 0$. Similar to the first argument, we have
 \begin{eqnarray*}   &  &
  \sum_{\rho(\lambda, \lambda')\ge s} h(\rho(\lambda,\lambda'))
 \nonumber\\
&  \hskip-0.08in \le &  \hskip-0.08in
\lim_{N\to \infty} h(N) \mu(B(\lambda,N))+  \sum_{n=s}^{N-1} \mu(B(\lambda,n)) (h(n)-h(n+1)\big)\nonumber\\
 &  \hskip-0.08in  \le & \hskip-0.08in  D_1({\mathcal G})
\lim_{N\to \infty} \Big(h(N) (N+1)^d+  \sum_{n=s}^{N-1}  (n+1)^d (h(n)-h(n+1)\big)\Big)\nonumber\\
& \hskip-0.08in  = &  \hskip-0.08in  D_1({\mathcal G})
\Big( (s+1)^d h(s)+\sum_{n=s+1}^\infty h(n) ((n+1)^d-n^d)\Big). 
\end{eqnarray*}
This proves  \eqref{jaffardpr.pf.eq2}. 
\end{proof}

\section{Matrices with polynomial off-diagonal decay}
\label{beurling.section}

Let ${\mathcal G}$ be a connected simple graph with Beurling dimension $d$.
For $1\le r\le \infty$ and $\alpha\ge 0$,  define 
\begin{equation}\label{beurling.def}
{\mathcal B}_{r, \alpha}({\mathcal G}):=\Big\{A=\big(a(\lambda, \lambda') \big)_{\lambda, \lambda' \in V}:  \  \  \|A\|_{{\mathcal B}_{r, \alpha}}<\infty\Big\},
\end{equation}
where $h_A(n)=\sup_{\rho(\lambda, \lambda')\ge n} |a(\lambda, \lambda')|, n\ge 0$,
and 
\begin{equation}\label{bralpha.norm}
\|A\|_{{\mathcal B}_{r, \alpha}}:=\left\{\begin{array}{ll}
\big(\sum_{n=0}^\infty |h_A(n)|^r (n+1)^{\alpha r+d-1} \big)^{1/r}  &  {\rm if} \ 1\le  r<\infty\\
 \sup_{n\ge 0} h_A(n) (n+1)^\alpha
  & {\rm if} \ r=\infty,\end{array}
  \right.
\end{equation}
\cite{baskakov14, beurling49, chengsun, jaffard90,  sunca11}.
We will use the abbreviated notation ${\mathcal B}_{r, \alpha}$ instead of
${\mathcal B}_{r, \alpha}({\mathcal G})$ if there is no confusion.
The commutative subalgebra
\begin{equation}\label{classicalbeurling.def}A^*:=\Big\{ \big(a(k-k')\big)_{k, k'\in \ZZ}: \ \sum_{n=0}^\infty \sup_{|k|\ge n} |a(k)|<\infty\Big\}
\end{equation}
of the class ${\mathcal B}_{r, \alpha}({\mathcal G})$  with    $r=1,\alpha=0$ and ${\mathcal G}={\mathcal Z}$  was introduced  by Beurling to study contraction
 of periodic functions  \cite{beurling49}. The set  ${\mathcal B}_{r, \alpha}({\mathcal G})$ with
 $r=+\infty,\alpha\ge 0$ and ${\mathcal G}={\mathcal Z}$
  is the Jaffard class  ${\mathcal J}_\alpha({\mathcal Z})$ of matrices with polynomial off-diagonal decay (\cite{chengsun, jaffard90}), since
$$\|A\|_{{\mathcal J}_\alpha}:=\sup_{i,j\in \ZZ} |a(i,j)| (1+|i-j|)^\alpha= \|A\|_{{\mathcal B}_{\infty, \alpha}} \ \ {\rm for} \ A:=\big(a(i,j) \big)_{i,j \in \ZZ}.$$
 The set ${\mathcal B}_{r, \alpha}({\mathcal G})$ with  $1\le r<\infty,\alpha\ge 0$ and ${\mathcal G}={\mathcal Z}$ is defined in \cite{sunca11}
 to contain all matrices  $A=(a(i,j))_{i,j\in \ZZ}$ with  
  \begin{equation} \label{oldbralpha.norm}
\|A\|_{{\mathcal B}_{r, \alpha}}^*=\Big(\sum_{n=0}^\infty \Big(\sup_{|i-j|\ge n} |a(i,j)| (1+|i-j|)^\alpha\Big)^r\Big)^{1/r}<\infty.
\end{equation}
We remark that norms in \eqref{bralpha.norm} and \eqref{oldbralpha.norm}  
 are equivalent to each other,
\begin{equation} \label{normequivalence.10}
\|A\|_{{\mathcal B}_{r, \alpha}}\le \|A\|_{{\mathcal B}_{r, \alpha}}^*\le 2^{2(\alpha+1/r)} \|A\|_{{\mathcal B}_{r, \alpha}}\  \ {\rm for \ all} \ A\in {\mathcal B}_{r, \alpha}({\mathcal Z}).
\end{equation}
The first inequality  in \eqref{normequivalence.10} follows immediately from
\eqref{bralpha.norm} and \eqref{oldbralpha.norm}, while the second estimate holds because for  any $A:=\big(a(i,j) \big)_{i,j \in \ZZ}\in
{\mathcal B}_{r, \alpha}({\mathcal Z})$ we have
\begin{eqnarray*}
(\|A\|_{{\mathcal B}_{r, \alpha}}^*)^r & \le &  \sup_{i,j\in \ZZ} |a(i,j)|^r (1+|i-j|)^{\alpha r} \\
 &  & + \sum_{l=0}^\infty  2^l
\Big(\sup_{|i-j|\ge 2^l} |a(i,j)| (1+|i-j|)^\alpha\Big)^r\\
& \le &   (h_A(0))^r
+\sum_{m=0}^\infty \big( h_A(2^m)\big)^r 2^{(m+1)\alpha r }\\
& &
 +
\sum_{l=0}^\infty  2^l \sum_{m=l}^\infty \big( h_A(2^m)\big)^r 2^{(m+1)\alpha r }\\
& \le &  (h_A(0))^r 
+ 2^{2\alpha r+2} \sum_{m=0}^\infty \big( h_A(2^m)\big)^r  2^{(m-1)(\alpha r+1)}\\
 & \le &  2^{2\alpha r+2} (\|A\|_{{\mathcal B}_{r, \alpha}})^r.
\end{eqnarray*}
Due to the above equivalence \eqref{normequivalence.10}, we follow the terminology in \cite{sunca11} to call ${\mathcal B}_{r, \alpha}({\mathcal G})$  as a Beurling class of matrices with polynomial off-diagonal decay.

\vskip0.1in

 Define the Schur norm $\|A\|_{{\mathcal S}}$ of  a  matrix $A:=\big(a(\lambda, \lambda') \big)_{\lambda, \lambda'\in V}$    by
\begin{equation} \|A\|_{{\mathcal S}}=
\max \Big(\sup_{\lambda\in V} \sum_{\lambda'\in V} |a(\lambda, \lambda')|,\
\sup_{\lambda'\in V} \sum_{\lambda\in V} |a(\lambda, \lambda')|\Big).
\end{equation}
Shown in the proposition below are some elementary properties of the Beurling class
${\mathcal B}_{r, \alpha}({\mathcal G})$, with their proofs postponed to the end of this section.

\begin{prop}
\label{beurling.prop}
Let $1\le q, r\le \infty, \alpha\ge 0$,  ${\mathcal G}:=(V,E)$ be a connected simple graph with Beurling dimension $d$.
Then the following statements hold.

\begin{itemize}

 \item [{(i)}]
${\mathcal B}_{1, 0}({\mathcal G})\subset {\mathcal S}\subset {\mathcal B}(\ell^q)$  and
\begin{equation} \label{beurling.prop.eq1}    \|A\|_{{\mathcal B}(\ell^q)}\le \|A\|_{{\mathcal S}}
\le d D_1({\mathcal G}) \|A\|_{{\mathcal B}_{1,0}} \ {\rm for \ all} \ A\in {\mathcal B}_{1, 0} ({\mathcal G}). \end{equation}

\item[{(ii)}] ${\mathcal B}_{r'', \beta}({\mathcal G})\subset {\mathcal B}_{r, \gamma}({\mathcal G})\subset {\mathcal B}_{r, \alpha}({\mathcal G})\subset {\mathcal B}_{r'', \alpha}({\mathcal G})$
    for all  $r''\ge r$, $\gamma\ge \alpha$ and $\beta> \gamma+d(1/r-1/r'')$. Moreover
    \begin{eqnarray}\label{beurling.prop.eq2}
     \|A\|_{{\mathcal B}_{r'', \alpha}} &\hskip-0.08in  \le & \hskip-0.08in   \|A\|_{{\mathcal B}_{r, \alpha}}\le  \|A\|_{{\mathcal B}_{r, \gamma}}
     \nonumber\\
     & \hskip-0.08in  \le & \hskip-0.08in   \Big(\frac{\beta-\gamma -(d-1)(1/r-1/r'')}{\beta-\gamma-d(1/r-1/r'')}\Big)^{1/r-1/r''}
      \|A\|_{{\mathcal B}_{r'', \beta}}
       \end{eqnarray}
      for  all $A\in {\mathcal B}_{r'', \beta} ({\mathcal G})$.

    \item[{(iii)}] ${\mathcal B}_{r, \alpha}({\mathcal G})$ is a Banach algebra if $\alpha> d(1-1/r)$, and
         \begin{eqnarray} \label{beurling.prop.eq3}
        \|AB\|_{{\mathcal B}_{r, \alpha}}   &\hskip-0.08in\le &\hskip-0.08in
         2^{\alpha+d/r} \big(\|B\|_{\mathcal S} \|A\|_{{\mathcal B}_{r, \alpha}} +
\|A\|_{\mathcal S} \|B\|_{{\mathcal B}_{r, \alpha}}\big)\nonumber\\
  &\hskip-0.08in \le &\hskip-0.08in  2^{\alpha+1+d/r} d D_1({\mathcal G}) \Big(\frac{\alpha-(d-1)(1-1/r)}{\alpha-d(1-1/r)}\Big)^{1-1/r}
           \|A\|_{{\mathcal B}_{r, \alpha}}\|B\|_{{\mathcal B}_{r, \alpha}}\end{eqnarray}
           for all $ A, B\in {\mathcal B}_{r, \alpha}({\mathcal G})$.

            \item[{(iv)}] ${\mathcal B}_{r, \alpha}({\mathcal G})$ is solid if $\alpha\ge 0$, i.e., if $A=(a(\lambda, \lambda'))_{\lambda, \lambda'\in V}$ and $B=(b(\lambda, \lambda'))_{\lambda, \lambda'\in V}$ satisfies
                $|a(\lambda, \lambda')|\le |b(\lambda, \lambda')|$ for all $\lambda, \lambda'\in V$, then $\|A\|_{{\mathcal B}_{r, \alpha}}\le \|B\|_{{\mathcal B}_{r, \alpha}}$.

\end{itemize}

\end{prop}

 A matrix $A=(a(\lambda, \lambda'))_{\lambda, \lambda'\in V}\in {\mathcal B}_{1, 0}({\mathcal G})$ (and hence in $B_{r, \alpha}({\mathcal G})$  with $1\le r\le \infty$ and $\alpha>d(1-1/r)$ by Proposition \ref{beurling.prop}) can be well approximated by  matrices with finite bandwidth,
\begin{equation}\label{AN.def}
A_N=\big(a(\lambda, \lambda') \chi_{[0,1]}(\rho(\lambda,\lambda')/N) \big)_{\lambda, \lambda' \in G}, \ N\ge 0,\end{equation}
in the Schur norm.
In particular, it follows from \eqref{jaffardpr.pf.eq2} and \eqref{beurling.prop.eq1}
that
\begin{equation}\label{bandapproximation}
\|A-A_N\|_{\mathcal S}\le   D_1({\mathcal G})
\Big( (N+2)^d h_A(N+1)+ d \sum_{n=N+2}^\infty h_A(n) (n+1)^{d-1}\Big),
\end{equation}
where $h_A(n)=\sup_{\rho(\lambda, \lambda)\ge n}|a(\lambda, \lambda')|, n\ge 0$.

\smallskip

By \eqref{beurling.prop.eq1} and \eqref{beurling.prop.eq2} in Proposition \ref{beurling.prop},
${\mathcal B}_{r,\alpha}$ with $1\le r\le \infty$ and  $\alpha> d(1-1/r)$ are Banach algebras, and they are  subalgebras of ${\mathcal B}(\ell^p), 1\le p\le \infty$,
 \begin{equation} \label{beurling.prop.eq1+}
 \|A\|_{{\mathcal B}(\ell^p)}
\le  \frac{\alpha- (d-1)(1-1/r)}{\alpha-d(1-1/r)} d D_1({\mathcal G}) \|A\|_{{\mathcal B}_{r,\alpha}} \ {\rm for \ all} \ A\in {\mathcal B}_{r, \alpha} ({\mathcal G}).
 \end{equation}
 Moreover,  following the argument in \cite{sunca11, suncasp05} and applying \eqref{jaffardpr.pf.eq2},
 we obtain that ${\mathcal B}_{r,\alpha}$ are differential subalgebras of ${\mathcal B}(\ell^2)$.

 \begin{prop}\label{subdiff.prop}  Let  ${\mathcal G}$ be a connected simple graph with Beurling dimension $d$,
 $1\le r\le \infty$, $\alpha>d(1-1/r)$, and set
 $$\theta_{r, \alpha}=\frac{2(\alpha-d+d/r)}{1+2\alpha-2d+2d/r}. $$ Then there exists an absolute constant $C$ such that
 \begin{equation}\label{subdiff.prop.eq1}
\|AB\|_{{\mathcal B}_{r,\alpha}}\le C  \|A\|_{{\mathcal B}_{r,\alpha}} \|B\|_{{\mathcal B}_{r,\alpha}} \Big (\Big(\frac{\|A\|_{{\mathcal B}(\ell^2)}}{\|A\|_{{\mathcal B}_{r,\alpha}}}\Big)^{\theta_{r, \alpha}} +
\Big(\frac{\|B\|_{{\mathcal B}(\ell^2)}}{\|B\|_{{\mathcal B}_{r,\alpha}}}\Big)^{\theta_{r, \alpha}}
\Big)
\end{equation}
hold for all $ A, B \in {\mathcal B}_{r,\alpha}({\mathcal G})$.
  \end{prop}

Applying \eqref{subdiff.prop.eq1} repeatedly and using the argument in \cite[Proposition 2.4]{sunaicm14}, we can find an absolute constant $C$
such that
\begin{eqnarray}\label{power.preeq111}
\|A_1\ldots A_n\|_{{\mathcal B}_{r, \alpha}} & \hskip-.08in \le & \hskip-0.08in  \Big( \max_{1\le k\le n} \|A_k\|_{{\mathcal B}(\ell^2)}\Big)^n\nonumber\\
& \hskip-.08in  & \hskip-0.08in \ \times
\Bigg( C\frac{\max_{1\le k\le n} \|A_k\|_{{\mathcal B}_{r, \alpha}}}{\max_{1\le k\le n} \|A_k\|_{{\mathcal B}(\ell^2)}}\Bigg)^{\frac{\theta_{r, \alpha}}{1+\theta_{r, \alpha}} n^{\log_2 (1+\theta_{r, \alpha})}}\end{eqnarray}
for all $A_1, \ldots, A_n\in {\mathcal B}_{r, \alpha}({\mathcal G}), n\ge 1$.
This together with
Proposition \ref{beurling.prop} implies that Banach algebras ${\mathcal B}_{r,\alpha}$  admit  norm-controlled inversions
in ${\mathcal B}(\ell^2)$.

\begin{cor}\label{wiener.cor}  Let  ${\mathcal G}$ be a connected simple graph with Beurling dimension $d$, and let
$1\le r\le \infty$ and $\alpha>d(1-1/r)$.
Then matrices in the Banach algebra ${\mathcal B}_{r,\alpha}({\mathcal G})$ admit norm-controlled inversions
in ${\mathcal B}(\ell^2)$.
\end{cor}

\begin{proof}  Take $A\in  {\mathcal B}_{r,\alpha}({\mathcal G})$ with $A^{-1}\in {\mathcal B}(\ell^2)$. Set $B= I- A^*A /\|A\|_{{\mathcal B}(\ell^2)}^2$. One may verify that
\begin{equation}\label{wiener.cor.pf.eq1}
\|B\|_{{\mathcal B}(\ell^2)}\le 1-  (\kappa(A))^{-2}<1 \ {\rm and} \ \|B\|_{{\mathcal B}_{r, \alpha}}\le  1+ \|A\|_{{\mathcal B}(\ell^2)}^{-2} \|A^*A\|_{{\mathcal B}_{r, \alpha}},  \end{equation}
where $\kappa (A)=
\|A\|_{{\mathcal B}(\ell^2)}\|A^{-1}\|_{{\mathcal B}(\ell^2)}$.
Therefore  by  \eqref{beurling.prop.eq3}, \eqref{power.preeq111} and \eqref{wiener.cor.pf.eq1}, we obtain
\begin{eqnarray*}
\|A^{-1} \|_{{\mathcal B}_{r, \alpha}}  & \hskip-0.08in = & \hskip-0.08in \| (A^*A)^{-1} A^*\|_{{\mathcal B}_{r, \alpha}}\le C
\|A^*\|_{{\mathcal B}_{r, \alpha}} \|A\|_{{\mathcal B}(\ell^2)}^{-2} \sum_{n=0}^\infty \|B^n\|_{{\mathcal B}_{r, \alpha}}\nonumber\\
& \hskip-0.08in \le &  \hskip-0.08in   C \|A\|_{{\mathcal B}_{r, \alpha}} \|A\|_{{\mathcal B}(\ell^2)}^{-2}
\sum_{n=0}^\infty
\big( 1-  (\kappa(A))^{-2}\big)^n\nonumber\\
& & \hskip-0.08in \times
\Bigg( C\frac{1+ \|A\|_{{\mathcal B}(\ell^2)}^{-2} \|A\|_{{\mathcal B}_{r, \alpha}}^2}{1-  (\kappa(A))^{-2}}\Bigg)^{\frac{\theta_{r, \alpha}}{1+\theta_{r, \alpha}} n^{\log_2 (1+\theta_{r, \alpha})}}<\infty,
\end{eqnarray*}
where $C$ is an absolute constant independent of the matrix $A$.
\end{proof}

\smallskip
We conclude this section with a proof of Proposition \ref{beurling.prop}.

\begin{proof} [Proof of Proposition \ref{beurling.prop}] (i). The first inequality in \eqref{beurling.prop.eq1} is well known.
Take
 $A:=\big(a(\lambda, \lambda') \big)_{\lambda, \lambda'\in V}\in {\mathcal B}_{1, 0}({\mathcal G})$.
Then it follows from Proposition \ref{graph.pr} that
\begin{eqnarray*}
& & \sum_{\lambda'\in V} |a(\lambda, \lambda')|
 \le    \sum_{\lambda'\in V} h_A(\rho(\lambda, \lambda'))\\
&\hskip-0.08in \le &\hskip-0.08in
D_1({\mathcal G})\big( h_A(0)+ d\sum_{n=1}^\infty h_A(n) (n+1)^{d-1}\big)
\le d D_1({\mathcal G}) \|A\|_{{\mathcal B}_{1, 0}},
\end{eqnarray*}
where
 $h_A(n)=\sup_{\rho(\lambda, \lambda')\ge n} |a(\lambda, \lambda')|, n\ge 0$.
 This proves the second estimate in \eqref{beurling.prop.eq1}.

 (ii).\  The conclusion is obvious for $r=r''$. Then it remains to prove \eqref{beurling.prop.eq2}
 for $r<r''$. The first inequality in \eqref{beurling.prop.eq2}  follows from the embedding property for weighted sequence spaces, and the second one
is obvious.
 Now we prove the third inequality in \eqref{beurling.prop.eq2}.
For any
$A:=\big(a(\lambda, \lambda') \big)_{\lambda, \lambda'\in V}\in {\mathcal B}_{r', \beta}({\mathcal G})$ with $r''=\infty$, we have
\begin{equation*}
\|A\|_{{\mathcal B}_{r, \gamma}}^r   \le     \|A\|_{{\mathcal B}_{\infty, \beta}}^r
\sum_{n=0}^\infty (n+1)^{(\gamma-\beta)r+d-1} \le  \frac{\beta-\gamma -d/r+1/r}{\beta-\gamma-d/r} \|A\|_{{\mathcal B}_{\infty, \beta}}^r.
\end{equation*}
This proves the third inequality in \eqref{beurling.prop.eq2} with $r''=\infty$.
 We can use similar argument to prove the third inequality in \eqref{beurling.prop.eq2} with $1\le r''<\infty$.

(iii).\  We follow the argument in \cite{sunca11} where the conclusion with ${\mathcal G}={\mathcal Z}^d$ is proved.
Clearly $\|\cdot\|_{{\mathcal B}_{r, \alpha}}$ is a norm. Then it suffices to prove \eqref{beurling.prop.eq3}. Take $A:=\big(a(\lambda, \lambda') \big)_{\lambda, \lambda'\in V}$ and $B:=\big(b(\lambda, \lambda') \big)_{\lambda, \lambda'\in V}\in {\mathcal B}_{r', \beta}({\mathcal G})$, and write $AB:=\big(c(\lambda, \lambda') \big)_{\lambda, \lambda'\in V}$.
Then  for all $\lambda, \lambda'\in V$ we have
\begin{eqnarray*}
 |c(\lambda, \lambda')|
 & \le &  \sum_{\lambda^{\prime\prime}\in V} |a(\lambda, \lambda^{\prime\prime})| |b(\lambda^{\prime\prime}, \lambda')|\nonumber\\
& \le &
h_A(\lfloor \rho(\lambda, \lambda')/2\rfloor) \sum_{\lambda^{\prime\prime}\in V}|b(\lambda^{\prime\prime}, \lambda')|\nonumber\\
& & \quad +
h_B(\lfloor \rho(\lambda, \lambda')/2\rfloor) \sum_{\lambda^{\prime\prime}\in V}|a(\lambda, \lambda^{\prime\prime})|\nonumber\\
& \le & \|B\|_{\mathcal S} h_A(\lfloor \rho(\lambda, \lambda')/2\rfloor)+ \|A\|_{\mathcal S} h_B(\lfloor \rho(\lambda, \lambda')/2\rfloor),
\end{eqnarray*}
where $h_A(n)=\sup_{\rho(\lambda, \lambda')\ge n} |a(\lambda, \lambda')|$ and $h_B(n)=\sup_{\rho(\lambda, \lambda')\ge n} |b(\lambda, \lambda')|, n\in \ZZ_+$.
Therefore
\begin{eqnarray}  \label{beurling.prop.pf.eq1}
\|AB\|_{{\mathcal B}_{\infty, \alpha}} &\hskip-0.08in \le & \hskip-0.08in
\|B\|_{\mathcal S} \sup_{n\ge 0} h_A(\lfloor n/2\rfloor) (n+1)^\alpha+
\|A\|_{\mathcal S} \sup_{n\ge0} h_B(\lfloor n/2\rfloor) (n+1)^\alpha\nonumber\\
& \hskip-0.08in \le & \hskip-0.08in 2^\alpha \big(\|B\|_{\mathcal S} \|A\|_{{\mathcal B}_{\infty, \alpha}} +
\|A\|_{\mathcal S} \|B\|_{{\mathcal B}_{\infty, \alpha}}\big)
\end{eqnarray}
for $r=\infty$, and
\begin{eqnarray}  \label{beurling.prop.pf.eq2}
\|AB\|_{{\mathcal B}_{r, \alpha}} & \hskip-0.08in \le & \hskip-0.08in \|B\|_{\mathcal S} \Big(\sum_{n=0}^\infty |h_A(\lfloor n/2\rfloor)|^r (n+1)^{\alpha r+d-1}\Big)^{1/r}\nonumber\\
& & \hskip-0.08in  +
\|A\|_{\mathcal S} \Big(\sum_{n=0}^\infty |h_B(\lfloor n/2\rfloor)|^r (n+1)^{\alpha r+d-1}\Big)^{1/r}\nonumber\\
& \hskip-0.08in  \le & \hskip-0.08in  2^{\alpha+d/r} \big(\|B\|_{\mathcal S} \|A\|_{{\mathcal B}_{r, \alpha}} +
\|A\|_{\mathcal S} \|B\|_{{\mathcal B}_{r, \alpha}}\big)
\end{eqnarray}
for $1\le r<\infty$. This proves the first inequality in \eqref{beurling.prop.eq3}.
The second inequality in  \eqref{beurling.prop.eq3} follows from \eqref{beurling.prop.eq1}, \eqref{beurling.prop.eq2} and the first estimate in
\eqref{beurling.prop.eq3}.

(iv). \ The solidness follows immediately from the definition \eqref{beurling.def} of the Beurling class ${\mathcal B}_{r, \alpha}({\mathcal G})$.
\end{proof}

\bigskip

\section{$\ell^p$-stability bound control}
\label{stability.section}

In this section, we prove the following result on lower $\ell^p$-stability bounds  of matrices in the Beurling class for different  exponent $1\le p\le \infty$.

\begin{thm}\label{stability.thm}
Let $1\le p, q, r\le \infty$, $r'=r/(r-1)$,  ${\mathcal G}$ be a connected simple graph  with Beurling dimension $d$,
the  counting measure $\mu$ on ${\mathcal G}$ have the doubling property (\ref{doubling}),
and let $A\in {\mathcal B}_{r,\alpha}({\mathcal G}) $ for some $ \alpha > d/r'$.
If  there exists a positive constant $A_p$ such that
 \begin{equation}\label{stability.thm.eq0}
 \|Ac\|_p\ge  A_p \|c\|_p \ \ {\rm for \ all}   \ c\in \ell^p,
 \end{equation}
 then there exists  a positive constant $A_q$ such that 
\begin{equation}\label{stability.thm.eq1}
 \|Ad\|_q \ge A_q\|d\|_q \ \ {\rm for \ all}\ d\in \ell^q.
\end{equation}
Moreover,
there exists an absolute constant $C$, independent of matrices $A\in {\mathcal B}_{r, \alpha}({\mathcal G})$ and  exponents $1\le p, q\le \infty$, such that the lower
$\ell^q$-stability bound $A_q$ in \eqref{stability.thm.eq1}  satisfies
\begin{equation}\label{stability.thm.eq2}
\frac{\|A\|_{{\mathcal B}_{r, \alpha}}}{A_q}
 \le  C  \left\{\begin{array}
{ll} \Big(\frac{\|A\|_{{\mathcal B}_{r, \alpha}}}{A_p}\Big)^{(1+\theta(p, q))^{K_0}}
 & {\rm if} \ \alpha\ne 1+d/r'\\
\Big(\frac{\|A\|_{{\mathcal B}_{r, \alpha}}}{A_p}
\ln \big(1+\frac{\|A\|_{{\mathcal B}_{r, \alpha}}}{A_p}\big)\Big)^{(1+\theta(p, q))^{K_0}}
  &  {\rm if} \ \alpha=1+d/r',
  \end{array}\right.
\end{equation}
where
 $$\theta(p, q)=\frac{d|1/p-1/q|}{K_0 \min(\alpha-d/r', 1)-d|1/p-1/q|}$$
 and  $K_0$ is a positive integer with
$$K_0> \frac{d}{\min(\alpha-d/r', 1)}.$$
\end{thm}

To prove Theorem \ref{stability.thm}, we introduce  a
truncation operator $\chi_\lambda^N$ and its smooth version
 $\Psi_{\lambda}^{N}$ by
\begin{equation}\label{chi.def}
 \chi_\lambda^N:\ \big(c(\lambda')\big)_{\lambda'\in V}\longmapsto \big ( \chi_{[0,1]}(\rho(\lambda,\lambda')/N)  c(\lambda') \big)_{\lambda'\in V}
\end{equation}
and
\begin{equation}\label{Psi.def}
 \Psi_\lambda^N :\  \big(c(\lambda')\big)_{\lambda'\in V}\longmapsto \big( \psi_0(\rho(\lambda,\lambda')/N)  c(\lambda') \big)_{\lambda'\in V},
\end{equation}
 where $\psi_0$ is    the trapezoid function  given by
\begin{equation*}
\psi_0(t)=\left\{\begin{array}{ll}
1 & {\rm if} \ |t|\le 1/2\\
2-2|t| & {\rm if} \ 1/2<|t|\le 1\\
0 & {\rm if} \ |t|> 1.
\end{array}
\right.
\end{equation*}
The truncation operator $\chi_\lambda^N$ and its smooth version $\Psi_\lambda^N$  localize a vector to the $N$-neighborhood of
the vertex $\lambda$, and it can also be considered as  diagonal matrices  with diagonal entries $\chi_{[0,1]}(\rho(\lambda,\lambda')/N)=\chi_{B(\lambda,N)}(\lambda')$ and $\psi_0(\rho(\lambda,\lambda')/N), \lambda'\in V$, respectively.
Our proof of Theorem \ref{stability.thm}
 depends on the estimate \eqref{stability.lem.pf.eq9-} for the commutator between
a matrix in the Beurling algebra and  the truncation operator $\Psi_\lambda^N$. Similar estimate has been used by Sj\"ostrand  in \cite{sjostrand} to establish inverse-closedness of the Baskakov-Gohberg-Sj\"ostrand subalgebra in ${\mathcal B}(\ell^2)$.

To prove Theorem \ref{stability.thm}, we  recall {\em maximal $N$-disjoint subsets}
 $V_N\subset V, N\ge 1$, which means that
\begin{equation}\label{max1}
B(\lambda, N)\cap\big(\cup_{\lambda_m \in V_N} B(\lambda_m,N) \big) \ne \emptyset
\ \  \text{ for all } \lambda \in V
\end{equation}
and
\begin{equation}\label{max2}
B(\lambda_m,N) \cap B(\lambda_n, N) =\emptyset \ \ \text{ for all distinct vertices } \lambda_m, \lambda_n \in V_N.
\end{equation}
We call vertices in a maximal $N$-disjoint set as fusion vertices \cite{chengsun}.
For a maximal $N$-disjoint set $V_N$,  the $N$-neighborhoods  $B(\lambda_m, N), \lambda_m\in V_N$, centered at fusion vertices have no common  vertices by \eqref{max2}. It is shown in \cite{chengsun} that
the $(2N)$-neighborhood  $B(\lambda_m, 2N), \lambda_m\in V_N$, is a covering of the set $V$.

\begin{prop}\label{prop-counting}
Let ${\mathcal G}:=(V, E)$ be a connected simple graph and  $\mu$ have the doubling property (\ref{doubling}).
If $V_N$  is a maximal $N$-disjoint subset of $V$, 
 then
\begin{eqnarray}\label{counting}
1 & \hskip-0.08in \le &  \hskip-0.08in  \inf_{\lambda \in V} \sum_{\lambda_m \in V_N}\chi_{B(\lambda_m, N')}(\lambda)
\nonumber\\
&\hskip-0.08in  \le &  \hskip-0.08in
\sup_{\lambda \in V} \sum_{\lambda_m \in V_N}\chi_{B(\lambda_m, N')}(\lambda)
\le (D_0({\mathcal G}))^{\lceil \log_2(2N'/N+1)\rceil} 
\end{eqnarray}
for all $N'\ge 2N.$
\end{prop}

To prove Theorem \ref{stability.thm}, we first establish its weak version,  the equivalence between $\ell^p$ and $\ell^q$-stabilities of a matrix with small $|1/p-1/q|$.

\begin{lem}\label{stability.lem}
Let $p, r, r',  d, \alpha,  {\mathcal G}, A$ be as in Theorem \ref{stability.thm}.
If $1\le q\le \infty$ satisfies
\begin{equation}\label{stability.lem.eq1}
d|1/p-1/q|< \min (\alpha-{d}/{r'}, 1),
\end{equation}
then
$A$ has $\ell^q$-stability. Furthermore there exists an absolute constant $C$, independent of matrices $A\in {\mathcal B}_{r, \alpha}({\mathcal G})$ and exponents $1\le p, q\le \infty$, such that the optimal
lower $\ell^q$-stability bound $A_q$ of the matrix $A$ satisfies
\begin{equation}\label{stability.lem.eq2}
A_q \ge  C  A_p 
\times
 \left\{ \begin{array}{ll} \Big(\frac{\|A\|_{{\mathcal B}_{r, \alpha}}}{A_p}\Big)^{-\theta_1(p,q)} & {\rm if} \ \alpha \ne d/r'+1\\
 \Big(\frac{\|A\|_{{\mathcal B}_{r, \alpha}}}{A_p}
\ln \Big(1+\frac{\|A\|_{{\mathcal B}_{r, \alpha}}}{A_p}\Big)
\Big)^{-\theta_1(p,q)} & {\rm if} \ \alpha = d/r'+1,\end{array}\right.
\end{equation}
where
$$\theta_1(p, q)=\frac{d|1/p-1/q|}{ \min(\alpha-d/r', 1)-d|1/p-1/q|}.$$
\end{lem}

\begin{proof}
Let $N\ge 2$ be a positive integer chosen later,  $V_N$ be a maximal $N$-disjoint set   of fusion vertices  satisfying
\eqref{max1} and \eqref{max2},
and let $\Psi_\lambda^N, \lambda\in V$, be the localization operators
 in \eqref{Psi.def}.
Take $c=\big(c(\lambda)\big)_{\lambda\in V}\in \ell^q$.
Applying the covering property \eqref{counting} of $\{B(\lambda_m, 2N), \lambda_m\in V_N\}$, we have
$$ 
\|c\|_{q} \le \big\|\big(
\| \Psi_{\lambda_m}^{4N}c\|_{q}\big)_{\lambda_m\in V_N}\big\|_q.
$$
Combining it  with  the polynomial growth property \eqref{polynomialgrowth} and
the norm equivalence between
$\|\Psi_{\lambda_m}^{4N} c\|_p$ and $\|\Psi_{\lambda_m}^{4N} c\|_q$, we obtain
  \begin{eqnarray}\label{stability.lem.pf.eq1} 
\|c\|_{q}  & \hskip-0.08in \le &  \hskip-0.08in
\big\|\big( (\mu(B(\lambda_m, 4N))   )^{(1/q-1/p)_+}\|\Psi_{\lambda_m}^{4N} c\|_p
\big)_{\lambda_m\in V_N}\big\|_q\nonumber\\
& \hskip-0.08in \le &  \hskip-0.08in  C N^{d(1/q-1/p)_+}
\big\|\big(\|\Psi_{\lambda_m}^{4N} c\|_p\big)_{\lambda_m\in V_N}\big\|_q.
\end{eqnarray}
Here in the proof, the capital letter $C$ denotes an absolute constant independent of matrices $A$,  sequences $c$, integers $N$, and exponents $p$ and $q$, which is
not necessarily the same at each occurrence.

For $\lambda\in V$,  it follows from the $\ell^p$-stability  \eqref{stability.thm.eq0} for the matrix $A$ that
\begin{equation}\label{stability.lem.pf.eq2}
A_p \|\Psi_{\lambda}^{4N} c\|_p\le \|A\Psi_{\lambda}^{4N} c\|_p.
\end{equation}
Let  $A_N, N\ge 2$, be matrices with finite bandwidth  in \eqref{AN.def}.
Combining \eqref{stability.lem.pf.eq1} and \eqref{stability.lem.pf.eq2}, we get
\begin{eqnarray}\label{stability.lem.pf.eq3} A_p \|c\|_{q} & \hskip-0.08in\le & \hskip-0.08in
CN^{d(1/q-1/p)_+}
\big\|\big(\|A\Psi_{\lambda_m}^{4N} c\|_p\big)_{\lambda_m\in V_N}\big\|_q\nonumber\\
& \hskip-0.08in
\le  &  \hskip-0.08in CN^{d(1/q-1/p)_+}  \Big( \big\|\big(\|(A-A_N)\Psi_{\lambda_m}^{4N} c\|_p\big)_{\lambda_m\in V_N}\big\|_q
\nonumber \\&  &
\hskip-0.08in
+\big\|\big(\|[A_N, \Psi_{\lambda_m}^{4N}] c\|_p\big)_{\lambda_m\in V_N}\big\|_q+
\big\|\big(\|\Psi_{\lambda_m}^{4N}(A_N-A) c\|_p\big)_{\lambda_m\in V_N}\big\|_q
\nonumber\\
&  &
\hskip-0.08in  +\big\|\big(\|\Psi_{\lambda_m}^{4N}Ac\|_p\big)_{\lambda_m\in V_N}\big\|_q
 \Big),
\end{eqnarray}
where $[A_N, \Psi_{\lambda_m}^{4N}] =A_N \Psi_{\lambda_m}^{4N}-\Psi_{\lambda_m}^{4N}A_N$  is the commutator between $A_N$ and $\Psi_{\lambda_m}^{4N}$ (\cite{sjostrand, sunca11}).

For any $ d 
\in \ell^q$, we obtain from the support property for 
 $\Psi_{\lambda_m}^{2N}$, the equivalence between two norms  $\|\chi_{\lambda_m}^{4N} d\|_p$ and $\|\chi_{\lambda_m}^{4N} d\|_q$,  the polynomial growth property \eqref{polynomialgrowth} and the covering property in Proposition \ref{prop-counting} that
\begin{eqnarray*} 
 &\hskip-0.08in  & \hskip-0.08in  \big\|\big(\|\Psi_{\lambda_m}^{4N}d\|_{p}\big)_{\lambda_m \in V_{N}}\big\|_q \le  \big\|\big(\|\chi_{\lambda_m}^{4N}d\|_{p}\big)_{\lambda_m \in V_{N}}\big\|_q\\
&  \hskip-0.08in\le  & \hskip-0.08in \big\|\big( \|\chi_{\lambda_m}^{4N} d\|_q (\mu(B(\lambda_m, 4N))^{(1/p-1/q)_+}\big)_{\lambda_m \in V_{N}}\big\|_{q}\nonumber\\
&  \hskip-0.08in\le  & \hskip-0.08in C N^{d(1/p-1/q)_+}
\big\|\big( \|\chi_{\lambda_m}^{4N} d\|_q \big)_{\lambda_m \in V_{N}}\big\|_{q}\le    C N^{d(1/p-1/q)_+} \|d\|_q.
\end{eqnarray*}
This together with \eqref {beurling.prop.eq1} yields the following three estimates:
\begin{equation} \label{stability.lem.pf.eq5}
\big\|\big(\|\Psi_{\lambda_m}^{4N}Ac\|_p\big)_{\lambda_m\in V_N}\big\|_q
\le C N^{d(1/p-1/q)_+} \|Ac\|_q,
\end{equation}
\begin{eqnarray}\label{stability.lem.pf.eq6}
 & & \big\|\big(\|(A-A_N)\Psi_{\lambda_m}^{4N} c\|_p\big)_{\lambda_m\in V_N}\big\|_q
\le   \|A-A_N\|_{\mathcal S} \big\|\big(\|\Psi_{\lambda_m}^{4N} c\|_p\big)_{\lambda_m\in V_N}\big\|_q\nonumber\\
 & & \qquad \le  C N^{d(1/p-1/q)_+}
 \|A-A_N\|_{\mathcal S} \|c\|_q,
\end{eqnarray}
and
\begin{eqnarray}\label{stability.lem.pf.eq5+}
\big\|\big(\|\Psi_{\lambda_m}^{4N}(A_N-A) c\|_p\big)_{\lambda_m\in V_N}\big\|_q
& \hskip-0.08in\le &  \hskip-0.08in  C N^{d(1/p-1/q)_+} \|(A_N-A) c\|_q \nonumber\\
& \hskip-0.08in\le &  \hskip-0.08in   C N^{d(1/p-1/q)_+}
 \|A-A_N\|_{\mathcal S} \|c\|_q.
\end{eqnarray}

Applying similar argument, we obtain
\begin{eqnarray}\label{stability.lem.pf.eq7}
 &\hskip-0.08in  & \hskip-0.08in \big\|\big(\|[A_N, \Psi_{\lambda_m}^{4N}] c\|_p\big)_{\lambda_m\in V_N}\big\|_q\nonumber\\
 & \hskip-0.08in \le  &  \hskip-0.08in \big(\sup_{\lambda\in V} \| [A_N, \Psi_{\lambda}^{4N}]\|_{\mathcal S}\big)\big\|\big(\|\chi_{\lambda_m}^{5N} c\|_p\big)_{\lambda_m\in V_N}\big\|_q\nonumber\\
\qquad \quad  & \hskip-0.08in \le &  \hskip-0.08in  C N^{d(1/p-1/q)_+}
\big(\sup_{\lambda\in V} \| [A_N, \Psi_{\lambda}^{4N}]\|_{\mathcal S}\big) \|c\|_q.
\end{eqnarray}
Combining \eqref{stability.lem.pf.eq3}--\eqref{stability.lem.pf.eq7}, we get
\begin{eqnarray}\label{stability.lem.pf.eq8}
A_p\|c\|_q  & \hskip-0.08in \le &  \hskip-0.08in   C N^{d|1/p-1/q|}
\big(\|A-A_N\|_{\mathcal S}+\sup_{\lambda\in V} \| [A_N, \Psi_{\lambda}^{4N}]\|_{\mathcal S}\big) \|c\|_q\nonumber\\
& &  +C N^{d|1/p-1/q|} \|Ac\|_q.
\end{eqnarray}

For any $\lambda\in V$,  we have
\begin{eqnarray} \label{stability.lem.pf.eq9-}
\|[A_{N}, \Psi_{\lambda}^{4N}]\|_{{\mathcal S}}
& \hskip-0.08in = & \hskip-0.08in \Big\|\Big(a(\lambda', \lambda'') \chi_{[0,1]}\Big(\frac{\rho(\lambda', \lambda'')}{N}\Big)\nonumber\\
& & \quad \times
\Big(
\psi_0\Big(\frac{\rho(\lambda', \lambda)}{4N}\Big)- \psi_0\Big(\frac{\rho(\lambda'', \lambda)}{4N}\Big)
 \Big)_{\lambda', \lambda''\in V}\Big\|_{{\mathcal S}}
\nonumber \\
& \hskip-0.08in \le &  \hskip-0.08in \frac{1}{2N} \Big\|\Big(|a(\lambda', \lambda'')|\rho(\lambda', \lambda'')| \chi_{[0,1]}\Big(\frac{\rho(\lambda', \lambda'')}{N}\Big)
 \Big)_{\lambda', \lambda''\in V}\Big\|_{{\mathcal S}}
\nonumber \\
& \hskip-0.08in \le &  \hskip-0.08in  C
N^{-1} \sum_{n=0}^N h_A(n) (n+1)^{d},
\end{eqnarray}
where the last inequality follows from \eqref{sum2}.
Therefore for any $\lambda\in V$,
\begin{eqnarray}\label{stability.lem.pf.eq9}
 \quad  &  \quad &   \|[A_{N}, \Psi_{\lambda}^{4N}]\|_{{\mathcal S}} 
 \le   C N^{-1}  \|A\|_{{\mathcal B}_{r, \alpha}}
\Big(\sum_{n=0}^{N} (n+1)^{-(\alpha-1) r'+d-1}\Big)^{1/r'}\nonumber\\
\quad &\quad  \le & \hskip-0.08in C  \|A\|_{{\mathcal B}_{r, \alpha}}
  \times
\left\{\begin{array}{ll}   N^{-1} &  {\rm if}\  \alpha>1+d/r'\\
N^{-1} (\ln (N+1))^{1-1/r} & {\rm if} \ \alpha=1+d/r'\\
 N^{-\alpha+d/r'} & {\rm if} \ \alpha<1+d/r'.
\end{array}\right.
\end{eqnarray}

For the Schur norm of $A-A_N$, there exists an absolute constant $C_0$, independent of $N\ge 1$ and $A\in {\mathcal B}_{r, \alpha}$, such that
\begin{eqnarray}\label{stability.lem.pf.eq10}
\|A-A_N\|_{\mathcal S}  &  \hskip-0.08in\le & \hskip-0.08in   C
\Big( (N+2)^d h_A(N+1)+ d\sum_{n=N+2}^\infty h_A(n) (n+1)^{d-1}\Big)\nonumber\\
& \hskip-0.08in\le & \hskip-0.08in  C_0 \|A\|_{{\mathcal B}_{r, \alpha}}
N^{-\alpha+d/r'},
\end{eqnarray}
where the first inequality follows from  \eqref{bandapproximation}
and the second inequality is true because
 \begin{eqnarray*}  \sum_{n=0}^{N+1} (h_A(n))^r (n+1)^{\alpha r+d-1} & \hskip-0.08in\ge & \hskip-0.08in  (h_A(N+1))^r \int_0^{N+2} t^{\alpha r+d-1} dt\\
  & \hskip-0.08in = & \hskip-0.08in (\alpha r+d)^{-1}  (h_A(N+1))^r (N+2)^{\alpha r+d}\end{eqnarray*}
and
$$ \sum_{n=N+2}^\infty (n+1)^{-\alpha r'+d-1}  \le
\int_{N+2}^\infty t^{-\alpha r'+d-1} dt\le  \frac{1}{\alpha r'-d} (N+2)^{-\alpha r'+d}. $$
Combining \eqref{stability.lem.pf.eq8}, \eqref{stability.lem.pf.eq9} and \eqref{stability.lem.pf.eq10}, we obtain
\begin{eqnarray}\label{stability.lem.pf.eq11}
A_p\|c\|_q  & \hskip-0.08in \le & \hskip-0.08in  C_1 \|A\|_{{\mathcal B}_{r, \alpha}}
  N^{d|1/p-1/q|-\min(\alpha-d/r', 1)} \|c\|_q\nonumber \\
  & & + C N^{d|1/p-1/q|} \|Ac\|_q
\end{eqnarray}
if $\alpha\ne 1+d/r'$, and
\begin{eqnarray}\label{stability.lem.pf.eq12}
A_p\|c\|_q  & \hskip-0.08in \le &  \hskip-0.08in  C_1 \|A\|_{{\mathcal B}_{r, \alpha}}
  N^{d|1/p-1/q|-1} (\ln N)^{1/r'} \|c\|_q\nonumber\\
   & & \hskip-0.08in + C N^{d|1/p-1/q|} \|Ac\|_q
\end{eqnarray}
if $\alpha=1+d/r'$, where $C_1$ is an absolute constant independent of   matrices $A$, integers $N$ and  sequences $c$.

For $\alpha\ne 1+d/r'$, replacing $N$ in \eqref{stability.lem.pf.eq11} by
$$N_0=\Big \lceil \big( 2 C_1 \|A\|_{{\mathcal B}_{r, \alpha}}/ A_p\big)^{(\min(\alpha-d/r', 1)-d|1/p-1/q|)^{-1}}\Big\rceil, $$
we get from  \eqref{stability.lem.eq1} and \eqref{stability.lem.pf.eq12} that
\begin{eqnarray*}
A_p\|c\|_q  & \hskip-0.08in\le & \hskip-0.08in  \frac{A_p}{2}  \|c\|_q
 +  C \Big(\frac{  \|A\|_{{\mathcal B}_{r, \alpha}}}{ A_p}\Big)^{\theta_1(p, q)}
 \|Ac\|_q.
\end{eqnarray*}
 This proves \eqref{stability.lem.eq2} for $\alpha\ne 1+d/r'$.

 For $\alpha=1+d/r'$, set
 $$ C_2:=  \frac{8 C_1\|A\|_{{\mathcal B}_{r, \alpha}}}{(1-d|1/p-1/q|) A_p}\ge 8 $$
  and
  $$N_1:=\big\lfloor (C_2 \ln  C_2\big)^{(1-d|1/p-1/q|)^{-1}}\big\rfloor \ge \frac{1}{2}
  \big( C_2\ln C_2\big)^{(1-d|1/p-1/q|)^{-1}}.$$
 Then
 \begin{equation}\label{stability.lem.pf.eq13}
C_1  \|A\|_{{\mathcal B}_{r, \alpha}}
 N_1^{d|1/p-1/q|-1}\le \frac{  (1-d|1/p-1/q|)  A_p}{
    4 \ln C_2}
 \end{equation}
 and
 \begin{equation} \label{stability.lem.pf.eq14}
N_1\le   (C_2\ln C_2)^{(1-d|1/p-1/q|)^{-1}}
 \le  C_2^{2(1-d|1/p-1/q|)^{-1}}.
 \end{equation}
 Replacing $N$ in \eqref{stability.lem.pf.eq12} by $N_1$ and applying
 \eqref{stability.lem.pf.eq13} and \eqref{stability.lem.pf.eq14}, we obtain
 \begin{eqnarray*}
A_p\|c\|_q  & \hskip-0.08in \le &  \hskip-0.08in\frac{(1-d|1/p-1/q|) A_p}{4 \ln C_2}\Big (2 (1-d|1/p-1/q|)^{-1} \ln C_2\Big)^{1-1/r} \|c\|_q\nonumber\\
& & \hskip-0.08in + C
(C_2\ln C_2)^{\theta_1(p, q)} \|Ac\|_q\nonumber\\
&\hskip-0.08in \le  & \hskip-0.08in  \frac{A_p}{2}  \|c\|_q
 +  C \big( C_2\ln C_2\big)^{\theta_1(p, q)}
 \|Ac\|_q.
\end{eqnarray*}
This proves
\eqref{stability.lem.eq2} for $\alpha=1+d/r'$.
\end{proof}

Having the above technical lemma,  we  use a bootstrap approach to prove Theorem \ref{stability.thm}, cf. \cite{jaffard90, shincjfa09, sunca11}.

\begin{proof}[Proof of Theorem \ref{stability.thm}] Let  $K$ be a positive integer
 with
 \begin{equation*} 
  d|1/p-1/q|/K<\min(\alpha-d/r', 1).
 \end{equation*}
 Then $K\le K_0$.
Let  $\{p_k\}_{k=0}^K$ be a monotone sequence such that
$$p_0=p, p_K=q\ {\rm  and} \  |1/p_k-1/p_{k+1}|=|1/p-1/q|/K, 0\le k\le K-1.$$
Applying Lemma \ref{stability.lem} repeatedly, we conclude that $A$ has $\ell^{p_k}$-stability for all $1\le k\le K$. Moreover the lower $\ell^{p_k}$-stability bound $A_{p_k}$ satisfies
 \begin{equation}\label{stability.thm.pf.eq2}
 \frac{A_{p_k}}{A_{p_{k+1}}}\le
C 
 \left\{ \begin{array}{ll} \Big(\frac{\|A\|_{{\mathcal B}_{r, \alpha}}}{A_{p_k}}\Big)^{\theta_K(p, q)} & {\rm if} \ \alpha \ne d/r'+1\\
 \Big(\frac{\|A\|_{{\mathcal B}_{r, \alpha}}}{A_{p_k}}
\ln \Big(1+\frac{\|A\|_{{\mathcal B}_{r, \alpha}}}{A_{p_k}}\Big)
\Big)^{\theta_K(p, q)} & {\rm if} \ \alpha = d/r'+1\end{array}\right.
 \end{equation}
 for all $0\le k\le K-1$, where
 $$\theta_K(p, q)=\frac{d|1/p-1/q|}{ K\min(\alpha-d/r', 1)-d|1/p-1/q|}$$
 and $C$ is an absolute constant independent of $A\in {\mathcal B}_{r, \alpha}$.

 For $\alpha\ne 1+d/r'$, we obtain from \eqref{stability.thm.pf.eq2} that
 \begin{equation}  \label{stability.thm.pf.eq3}
 \frac{A_{p_{k+1}}}{\|A\|_{{\mathcal B}_{r, \alpha}}} \ge C
  \Big(\frac{ A_p}{\|A\|_{{\mathcal B}_{r, \alpha}}}\Big)^{(1+\theta_{K}(p, q))^{k+1}}, \ 0\le k\le K-1.
 \end{equation}
 This proves \eqref{stability.thm.eq2} for $\alpha\ne 1+d(1-1/r)$.

 For $\alpha=1+d/r'$,   it follows from
    \eqref{stability.thm.pf.eq2} that
$$ \frac{\|A\|_{{\mathcal B}_{r, \alpha}}}{A_{p_{k+1}}}\le  C \Big(\frac{\|A\|_{{\mathcal B}_{r, \alpha}}}{A_{p_k}}\Big)^{1+\theta_K(p, q)} \Big(\ln \Big(1+\frac{\|A\|_{{\mathcal B}_{r, \alpha}}}{A_{p_k}}\Big)\Big)^{\theta_K(p, q)}, \ 0\le k\le K-1.$$
Applying the above estimate repeatedly, we obtain
\begin{equation*}
\frac{\|A\|_{{\mathcal B}_{r, \alpha}}}{A_{p_{k}}}\le C \Big(\frac{\|A\|_{{\mathcal B}_{r, \alpha}}}{A_{p}}\Big)^{(1+\theta_K(p, q))^k}
 \Big(\ln \Big(1+\frac{\|A\|_{{\mathcal B}_{r, \alpha}}}{A_{p}}\Big)\Big)^{(1+\theta_K(p, q))^k-1}\end{equation*}
 by induction on $1\le k\le K$. This proves
 \eqref{stability.thm.eq2} for $\alpha=1+d/r'$.
\end{proof}

\section{Norm-controlled inversion}
\label{inversion.section}


 By Corollary \ref{wiener.cor},
  matrices in  Banach algebras ${\mathcal B}_{r,\alpha}$ with $1\le r\le \infty$ and $\alpha>d(1-1/r)$  admit norm-controlled inversions
in ${\mathcal B}(\ell^2)$.
In this section, we show that a polynomial can be selected to be the norm-controlled function  $h$ in \eqref{normcontrol.def} if the
the  counting measure  $\mu$ on the graph ${\mathcal G}$  is
normal.

\begin{thm}\label{inverse.thm}
Let  $1\le r\le \infty, r'=r/(r-1)$,  $\alpha> d/r'$,  ${\mathcal G}$ be a connected simple graph with Beurling dimension $d$
and normal  counting measure $\mu$,  and let
 $A\in {\mathcal B}_{r,\alpha}({\mathcal G})$ be invertible in ${\mathcal B}(\ell^2)$. Then
  there exists an absolute constant $C$, independent of $A$,
 such that
\begin{eqnarray}\label{inverse.thm.eq1}
\|A^{-1}\|_{{\mathcal B}_{r, \alpha}}  &\hskip-0.08in \le & \hskip-0.08in C
\|A^{-1}\|_{{\mathcal B}(\ell^2)}
(\|A^{-1}\|_{{\mathcal B}(\ell^2)}\|A\|_{{\mathcal B}_{r, \alpha}})^{(\alpha+d/r)/\min(\alpha-d/r', 1)}
\nonumber\\
& \hskip-0.08in  & \hskip-0.08in  \times \left\{\begin{array}{ll}
\hskip-0.08in 1 & \hskip-0.08in {\rm if} \ \alpha\ne 1+d/r'\\
\hskip-0.08in \big(\ln \big( \|A^{-1}\|_{{\mathcal B}(\ell^2)} \|A\|_{{\mathcal B}_{r, \alpha}}+1\big)\big)^{(d+1)/r'} & \hskip-0.08in
{\rm if} \ \alpha=1+d/r'.
\end{array}\right.
\end{eqnarray}
\end{thm}

For invertible matrices $A\in {\mathcal B}_{r, \alpha}({\mathcal G})$ with $\alpha>1+d/r'$,
it follows from  Theorem \ref{inverse.thm} that
\begin{equation}\label{inverse.cor.eq1}
\|A^{-1}\|_{{\mathcal B}_{r, \alpha}}  \le C
\|A^{-1}\|_{{\mathcal B}(\ell^2)}
(\|A^{-1}\|_{{\mathcal B}(\ell^2)}\|A\|_{{\mathcal B}_{r, \alpha}})^{\alpha+d/r}.
\end{equation}
A weak version of the above estimate, with the exponent $\alpha+d/r$ in \eqref{inverse.cor.eq1} replaced by a larger exponent $2\alpha+2+2/(\alpha-2)$, is established in \cite{grochenig14-1} for matrices in  the Jaffard algebra ${\mathcal J}_\alpha({\mathcal Z})={\mathcal B}_{\infty, \alpha}({\mathcal Z})$, where $r=\infty$.

The estimate  \eqref{inverse.cor.eq1} on norm-controlled inversion is almost optimal, as shown in the following example that for any $\epsilon>0$ there does not exist an absolute constant  $C_\epsilon$ such that
\begin{equation}\label{inverse.cor.eq2}
\|A^{-1}\|_{{\mathcal B}_{r, \alpha}}  \le C_\epsilon
\|A^{-1}\|_{{\mathcal B}(\ell^2)}
(\|A^{-1}\|_{{\mathcal B}(\ell^2)}\|A\|_{{\mathcal B}_{r, \alpha}})^{\alpha+d/r-1-\epsilon}.
\end{equation}

\begin{ex}\label{agamma.example} {\rm
Let $1\le r\le \infty, \alpha>1-1/r$ and ${\mathcal G}={\mathcal Z}^d$ with $d=1$.
For sufficiently small $\gamma \in (0, 1)$, define  $A_\gamma=(a_\gamma(i,j))_{i,j\in \ZZ}$
 by
\begin{equation}\label{agamma.eq1}
a_\gamma(i,j)=\left\{\begin{array} {ll} 1  & {\rm if} \ j=i\\
-e^{-\gamma} & {\rm if} \ j=i+1\\
0 & {\rm elsewhere}.
\end{array}\right.
\end{equation}
Then
\begin{equation}\label{aepsilon.eq1}
\|A_\gamma\|_{{\mathcal B}_{r, \alpha}}=  (1+  2^{\alpha r}e^{-\gamma r} )^{1/r}\in  [2^{\alpha-1}, 2^{\alpha+1}].
\end{equation}
 Observe that
$A_\gamma$ is invertible in ${\mathcal B}(\ell^2)$ and  its inverse is given by  $B_\gamma=(b_\gamma(i,j))_{i,j\in \ZZ}$, where
$$   b_\gamma(i,j)=\left\{\begin{array} {ll} 
e^{-(j-i)\gamma} & {\rm if} \ j\ge i\\
0 & {\rm elsewhere}.
\end{array}\right.
$$
Therefore  for sufficiently small $\gamma\in (0, 1)$, we have
\begin{equation} \|A^{-1}\|_{{\mathcal B}(\ell^2)}= (1-e^{-\gamma})^{-1}\in [\gamma^{-1}, 2 \gamma^{-1}]\end{equation}
and
\begin{eqnarray}\|A^{-1}\|_{{\mathcal B}_{r, \alpha}} & \hskip-0.08in = &  \hskip-0.08in\left\{\begin{array}{ll}
(\sum_{n=0}^\infty (n+1)^{\alpha r} e^{-n r \gamma}\big)^{1/r} & {\rm if} \ 1\le r<\infty\\
 \sup_{n\ge 0} (n+1)^\alpha e^{-n\gamma} & {\rm if} \ r=\infty
 \end{array}\right.\nonumber\\
 & \hskip-0.08in \in & \hskip-0.08in   \gamma^{-\alpha-1/r} [1, 2]\times
 \left\{\begin{array}{ll}
r ^{-\alpha-1/r} (\Gamma(\alpha r+1))^{1/r}  & {\rm if} \ 1\le r<\infty\\
 (\alpha/e)^\alpha & {\rm if} \ r=\infty,
 \end{array}\right.
 \end{eqnarray}
 where $\Gamma(s)=\int_0^\infty x^{s-1} e^{-x} dx$ is the Gamma function.
 Hence for sufficiently small $\gamma$, the left hand side of \eqref{inverse.cor.eq2}
is of order $\gamma^{-\alpha-1/r}$ and the right hand side of \eqref{inverse.cor.eq2}
is of order $\gamma^{-\alpha-1/r+\epsilon}$ for $\alpha>1+d/r'$. This proves
\eqref{inverse.cor.eq2}.
}\end{ex}

To prove Theorem \ref{inverse.thm},  we need a distribution property for
 fusion vertices of a maximal $N$-disjoint set.

 \begin{prop} \label{fusion.pr}
Let ${\mathcal G}:=(V,E)$ be a connected simple graph with Beurling dimension $d$ and normal  counting measure  $\mu$, and
let
$V_N, 1\le N\le {\rm diam}({\mathcal G})$, be maximal $N$-disjoint sets of  fusion vertices.
Then for all $\lambda\in V$,
\begin{equation}\label{fusion.pr.eq1}
\#\{\lambda_m\in V_N, \rho(\lambda_m, \lambda)\le NR\}\le \frac{D_1({\mathcal G})}{ D_2({\mathcal G})} (R+1)^d, \ 0\le R\le \frac{{\rm diam}({\mathcal G})}{N}+1,
\end{equation}
 and
 \begin{equation}\label{fusion.pr.eq2}
\#\{\lambda_m\in V_N, \rho(\lambda_m, \lambda)\le NR\}\ge \frac{D_2({\mathcal G})}{ D_1({\mathcal G})} \Big(\frac{R-2}{3}\Big)^d, \ 3\le R\le
\frac{{\rm diam}({\mathcal G})}{N}+1.
\end{equation}
\end{prop}

\begin{proof}
Take a vertex $\lambda\in V$ and a  nonnegative integer   $R$. Set
$$E=\{\lambda_m\in V_N: \  \rho(\lambda_m, \lambda)\le NR\}.$$
Then
\begin{equation*}
\cup_{\lambda_m\in E} B(\lambda_m, N)
\subset B(\lambda, N(R+1)).
\end{equation*}
This, together with   \eqref{polynomialgrowth}, \eqref{ahlfors.def} and \eqref{max2},
 implies that
\begin{eqnarray*}
D_2({\mathcal G})(N+1)^d \# E  &\hskip-0.08in  \le & \hskip-0.08in  \sum_{\lambda_m\in E} \mu(B(\lambda_m, N))\nonumber\\
& \hskip-0.08in = &\hskip-0.08in
 \mu\big(\cup_{\lambda_m\in E}B(\lambda_m, N)\big)
\le D_1({\mathcal G})(N(R+1)+1)^d.
\end{eqnarray*}
Hence the upper bound estimate in \eqref{fusion.pr.eq1} follows.

Take a vertex $\lambda\in V$ and an integer   $R\ge 3$.  Applying the covering property   \eqref{counting}, we have
$$ B(\lambda, (R-2)N)\subset \cup_{\lambda_m\in E} B(\lambda_m, 2N).$$
This together with  \eqref{polynomialgrowth} and \eqref{ahlfors.def}
implies that
\begin{equation*} 
D_2({\mathcal G}) ((R-2)N+1)^d\le D_1({\mathcal G})(2N+1)^d \# E.
\end{equation*}
Hence the lower bound estimate in \eqref{fusion.pr.eq2} follows.
\end{proof}

Let
$V_N, N\ge 1$, be maximal $N$-disjoint sets of  fusion vertices.
Let
$ {\mathcal B}_{r, \alpha; N}$ contain all matrices  $B:=(b(\lambda_m, \lambda_{m'}))_{\lambda_m, \lambda_{m'}\in V_N}$
with $\|B\|_{{\mathcal B}_{r, \alpha; N}}<\infty$,
where  $h_{B, N}(n)= \sup_{\rho(\lambda_m, \lambda_{m'})\ge Nn} |b(\lambda_m, \lambda_{m'})|, n\ge 0$, and
$$\|B\|_{{\mathcal B}_{r, \alpha; N}}=\left\{\begin{array}
{ll} \big(\sum_{n=0}^{\infty} (h_{B, N} (n))^r (n+1)^{\alpha r+d-1} \big)^{1/r}
  & {\rm if} \ 1\le r<\infty\\
   \sup_{n\ge 0}  h_{B, N}(n) (n+1)^\alpha & {\rm if} \ r=\infty,
\end{array}\right.
$$
  cf. the Beurling class ${\mathcal B}_{r, \alpha}({\mathcal G})$  in \eqref{bralpha.norm}.
Clearly $\|\cdot\|_{{\mathcal B}_{r, \alpha; N}}$ is a norm.
The next proposition states that $ {\mathcal B}_{r, \alpha; N}$ are Banach algebras.

\begin{prop}\label{VNalgebra.prop}
Let ${\mathcal G},  r, \alpha$ be as in Theorem \ref{inverse.thm}, and let
$V_N, N\ge 1$, be a  maximal $N$-disjoint set of  fusion vertices.
Then there exists an absolute constant
$C$, independent of integers $N\ge 1$,  such that
\begin{equation}\label{VNalgebra.lem.eq1}
\|AB\|_{{\mathcal B}_{r, \alpha; N}}\le  C
  \|A\|_{{\mathcal B}_{r, \alpha; N}}
 \|B\|_{{\mathcal B}_{r, \alpha; N}}\ \ {\rm for \ all} \  A, B\in {\mathcal B}_{r, \alpha; N}.
\end{equation}
\end{prop}

The above lemma can be proved by following the argument used in Proposition \ref{beurling.prop}. We omit the detailed proof here.

\smallskip

To prove Theorem \ref{inverse.thm}, we also need a technical lemma.

%
\begin{lem}\label{inverse.lem}
Let ${\mathcal G}, A, r, r',  \alpha$ be as in Theorem \ref{inverse.thm}, and let
$V_N, N\ge 2$, be maximal $N$-disjoint sets of  fusion vertices.
Then there exists an absolute constant $C_0$ independent of $A$ such that
\begin{eqnarray} \label{inverse.lem.eq2}
 \|\Psi_{\lambda_m}^{4N} c\|_{2} & \le &  2\|A^{-1}\|_{{\mathcal B}(\ell^2)} \Big(\|\Psi_{\lambda_m}^{4N} A c\|_{\ell^2}\nonumber\\
 & & \quad +
 \sum_{\lambda_{m'}\in V_N}
 \|\chi_{\lambda_m}^{5N} [\Psi_{\lambda_m}^{4N},  A]\chi_{\lambda_{m'}}^{4N}\|_{{\mathcal S}}
 \|\Psi_{\lambda_{m'}}^{4N} c\|_{2}\Big)
 \end{eqnarray}
for all vertices $\lambda_m\in V_N$, sequences $c\in \ell^2$ and
integers $N$ satisfying
\begin{equation}  \label{inverse.lem.eq1}
N^{\alpha-d/r'}\ge   2  C_0
 \|A^{-1}\|_{{\mathcal B}(\ell^2)}\|A\|_{{\mathcal B}_{r, \alpha}},
  \end{equation}
      where $ [\Psi_{\lambda_m}^{4N},  A]= \Psi_{\lambda_m}^{4N}A-A\Psi_{\lambda_m}^{4N}$ and $C_0$ is the constant in \eqref{stability.lem.pf.eq10}.
  Moreover, there exists an absolute constant $C$ independent of $A\in {\mathcal B}_{r, \alpha}({\mathcal G})$ such that
\begin{eqnarray} \label{inverse.lem.eq3}
  \Big(\sum_{n=0}^\infty (h_{A, N}(n))^{r} (n+1)^{\alpha r+d-1}\Big)^{1/r} & \hskip-0.08in\le & \hskip-0.08in
  C  \|A\|_{{\mathcal B}_{r, \alpha}} N^{-\min(\alpha-d/r', 1)} \nonumber\\
  & \hskip-0.08in & \hskip-0.08in
\times \left \{\begin{array}{ll} 1  & {\rm if} \ \alpha\ne 1+d/r'\\
 (\ln N)^{1/r'} & {\rm if} \ \alpha=1+d/r'
\end{array} \right.
 \end{eqnarray}
 if $1\le r<\infty$, and
\begin{equation} \label{inverse.lem.eq4}
\sup_{n\ge 0} h_{A, N}(n) (n+1)^\alpha \le C  \|A\|_{{\mathcal B}_{\infty, \alpha}} N^{-\min(\alpha-d, 1)}
 \left \{\begin{array}{ll}  1 & {\rm if} \ \alpha\ne d+1\\
\ln N & {\rm if} \ \alpha=d+1
\end{array} \right.
 \end{equation}
if $r=\infty$, where
  \begin{equation}\label{hAN.def}
  h_{A, N}(n)=\sup_{\rho(\lambda_m, \lambda_{m'})\ge Nn} \| \chi_{\lambda_m}^{5N}[\Psi_{\lambda_m}^{4N},  A]
\chi_{\lambda_{m'}}^{4N}\|_{{\mathcal S}},\  n\ge 0.
  \end{equation}
\end{lem}

\begin{proof} 
 We follow the argument in \cite{sjostrand, sunca11} where ${\mathcal G}=\ZZ^d$ with $d=1$.  Take
  $\lambda_m \in V_{N}, c:=(c(\lambda))_{\lambda\in V}\in \ell^2$, and  let $A_N$ be as in   \eqref{AN.def}.
  By the invertibility  on $A$, we have
\begin{eqnarray}\label{inverse.lm.pf.eq1}
  \|\Psi_{\lambda_m}^{4N} c\|_{2}   & \hskip-0.08in \le & \hskip-0.08in    \|A^{-1}\|_{{\mathcal B}(\ell^2)}
    \| A \Psi_{\lambda_m}^{4N} c\|_{2}\nonumber\\
&   \hskip-0.08in \le  & \hskip-0.08in
     \|A^{-1}\|_{{\mathcal B}(\ell^2)}  \| \Psi_{\lambda_m}^{4N} Ac\|_{2}+
     \|A^{-1}\|_{{\mathcal B}(\ell^2)}  \| [\Psi_{\lambda_m}^{4N},  A] c\|_{2}.
     \end{eqnarray}

     By the covering property in Proposition  \ref{prop-counting},
 $\Psi^{4N}:=\sum_{\lambda_{m'}\in V_N} \Psi_{\lambda_{m'}}^{4N}$ is a diagonal matrix
 with  bounded inverse, and
 \begin{equation}\label{inverse.lm.pf.eq2}  \|(\Psi^{4N})^{-1}\|_{{\mathcal B}(\ell^2)}\le 1.
 \end{equation}
  Therefore
      \begin{eqnarray} \label{inverse.lm.pf.eq3}
 \| [\Psi_{\lambda_m}^{4N}, A] c\|_{2}
  & \hskip-0.08in \le   &\hskip-0.08in
 \| \chi_{\lambda_m}^{5N} [\Psi_{\lambda_m}^{4N}, A]
  c\|_{2}
  +
    \|(I- \chi_{\lambda_m}^{5N})A\chi_{\lambda_m}^{4N}
    \Psi_{\lambda_m}^{4N} c\|_{2}\nonumber
    \\
 & \hskip-0.08in\le &\hskip-0.08in
 \sum_{\lambda_{m'}\in V_N}
   \| \chi_{\lambda_m}^{5N}[\Psi_{\lambda_m}^{4N}, A]
  (\Psi^{4N})^{-1} \Psi_{\lambda_{m'}}^{4N} c\|_{2}\nonumber\\
  & &
  +
    \|(I- \chi_{\lambda_m}^{5N})A 
    \Psi_{\lambda_m}^{4N} c\|_{2}\nonumber \\
  & \hskip-0.08in\le & \hskip-0.08in
 \sum_{\lambda_{m'}\in V_N}
   \|\chi_{\lambda_m}^{5N}  [\Psi_{\lambda_m}^{4N}, A]
 \chi_{\lambda_{m'}}^{4N}\|_{{\mathcal B}(\ell^2)}
 \|(\Psi^{4N})^{-1}\|_{{\mathcal B}(\ell^2)}
 \|\Psi_{\lambda_{m'}}^{4N} c\|_{2}\nonumber\\
 & &  +    \|(I- \chi_{\lambda_m}^{5N})A \chi_{\lambda_m}^{4N}\|_{{\mathcal S}}
   \| \Psi_{\lambda_m}^{4N} c\|_{2}\nonumber \\
   & \hskip-0.08in \le &  \hskip-0.08in
 \sum_{\lambda_{m'}\in V_N}
   \|\chi_{\lambda_m}^{5N}  [\Psi_{\lambda_m}^{4N}, A]
 \chi_{\lambda_{m'}}^{4N}\|_{{\mathcal S}}
 \|\Psi_{\lambda_{m'}}^{4N} c\|_{2}\nonumber\\
 & &
 +    \|A-A_N\|_{{\mathcal S}}
   \| \Psi_{\lambda_m}^{4N} c\|_{2},
 \end{eqnarray}
 where the last inequality follows from \eqref{inverse.lm.pf.eq2} and Proposition \ref{beurling.prop}, cf. \cite{sjostrand, sunca11}.
Combining \eqref{inverse.lm.pf.eq1} and \eqref{inverse.lm.pf.eq3}, and then using \eqref{stability.lem.pf.eq10} and \eqref{inverse.lem.eq1}, we complete the proof of the upper bound estimate \eqref{inverse.lem.eq2} for $\|\Psi_{\lambda_m}^{4N}c\|_2$.

Write $A=( a(\lambda, \lambda'))_{\lambda, \lambda'\in V}$ and define
$h_A(n)=\sup_{\rho(\lambda, \lambda')\ge n} |a(\lambda, \lambda')|$. For $\lambda_m, \lambda_{m'}\in V_N$ with $\rho(\lambda_m, \lambda_{m'})\ge  16N$,
we obtain from \eqref{polynomialgrowth}  and the supporting property for $\Psi_{\lambda_m}^{4N}$ that
\begin{equation}\label{inverse.lm.pf.eq4}
\|\chi_{\lambda_m}^{5N} [\Psi_{\lambda_m}^{4N}, A]\chi_{\lambda_{m'}}^{4N}\|_{{\mathcal S}}=
\| \Psi_{\lambda_m}^{4N} A\chi_{\lambda_{m'}}^{4N}\|_{{\mathcal S}}
\le C h_A(\rho(\lambda_m, \lambda_{m'})/2) N^d.
\end{equation}
For $\lambda_m, \lambda_{m'}\in V_N$ with $\rho(\lambda_m, \lambda_{m'})<16N$,
\begin{eqnarray}\label{inverse.lm.pf.eq5}
 & & \|\chi_{\lambda_m}^{5N} [\Psi_{\lambda_m}^{4N}, A]\chi_{\lambda_{m'}}^{4N}\|_{{\mathcal S}}\nonumber\\
& = &
\Big\|\Big( \chi_{[0,5N]}(\rho(\lambda, \lambda_m))
  a(\lambda, \lambda') \chi_{[0, 4N]}(\rho(\lambda', \lambda_{m'})) \nonumber\\
  & & \quad \times
  \Big(\psi_0\Big(\frac{\rho(\lambda, \lambda_{m})}{4N}\Big)-
 \psi_0\Big(\frac{\rho(\lambda', \lambda_{m})}{4N}\Big)\Big)
 \Big)_{\lambda, \lambda'\in V}\Big\|_{\mathcal S}\nonumber\\
 & \le &
C N^{-1} \sum_{n=0}^{25N} h_A(n) (n+1)^{d}\nonumber\\
 & \le &  C  \|A\|_{{\mathcal B}_{r, \alpha}}
\left\{\begin{array}{ll} N^{-1} &  {\rm if}\  \alpha>1+d/r'\\
N^{-1} (\ln N)^{1/r'} & {\rm if} \ \alpha=1+d/r'\\
N^{-\alpha+d/r'} & {\rm if} \ \alpha<1+d/r',
\end{array}\right.
\end{eqnarray}
where the first inequality follows from \eqref{sum2}, and the last estimate is obtained by applying a H\"older inequality, cf. \eqref{stability.lem.pf.eq9}.
Combining \eqref{inverse.lm.pf.eq4} and \eqref{inverse.lm.pf.eq5} proves
\eqref{inverse.lem.eq3} and \eqref{inverse.lem.eq4}.
\end{proof}

Now we start to the proof of Theorem \ref{inverse.thm}.

\begin{proof}[Proof of Theorem \ref{inverse.thm}]
Let  $N\ge 2$ be  chosen later. Define
$V_{A, N}:=(V_{A,N}(\lambda_m, \lambda_{m'}))_{\lambda_m, \lambda_{m'}\in V_N}$
and write
$(V_{A, N})^l= (V_{A,N}^l(\lambda_m, \lambda_{m'}))_{\lambda_m, \lambda_{m'}\in V_N} $,
where
\begin{equation*}\label{VAN1}
V_{A,N}(\lambda_m, \lambda_{m'}) = 2\|A^{-1}\|_{{\mathcal B}(\ell^2)}
 \|\chi_{\lambda_m}^{5N} [\Psi_{\lambda_m}^{4N}, A]\chi_{\lambda_{m'}}^{4N}\|_{{\mathcal S}}
 \end{equation*}
 Then we obtain from  Proposition \ref{VNalgebra.prop}  and Lemma \ref{inverse.lem} that
 \begin{eqnarray}\label{inverse.thm.pf.eq1}
 \|V_{A, N}\|_{{\mathcal B}_{r, \alpha; N}} & \le &
  D_3  \|A^{-1}\|_{{\mathcal B}(\ell^2)} \|A\|_{{\mathcal B}_{r, \alpha}} N^{-\min(\alpha-d/r', 1)} \nonumber\\
  & &\quad
\times \left \{\begin{array}{ll} 1  & {\rm if} \ \alpha\ne 1+d/r'\\
 (\ln N)^{1/r'} & {\rm if} \ \alpha=1+d/r',
\end{array} \right.
 \end{eqnarray}
 and
 \begin{equation} \label{inverse.thm.pf.eq2}
 \|(V_{A, N})^l\|_{{\mathcal B}_{r, \alpha; N}}\le D_4^{l-1}
\|(V_{A, N})\|_{{\mathcal B}_{r, \alpha; N}}^l
 \end{equation}
where $D_3, D_4$ are absolute constants independent of matrices $A$ and integers $N$ and $l$.

Let $N_2\ge 2$ be the minimal integer satisfying \eqref{inverse.lem.eq1}
and
\begin{eqnarray}  \label{inverse.thm.pf.eq3}
1 & \hskip-0.08in \ge &  \hskip-0.08in 4 D_3D_4   \|A^{-1}\|_{{\mathcal B}(\ell^2)} \|A\|_{{\mathcal B}_{r, \alpha}} N_2^{-\min(\alpha-d/r', 1)} \nonumber\\
  & &\qquad
\times \left \{\begin{array}{ll} 1  & {\rm if} \ \alpha\ne 1+d/r'\\
 (\ln N_2)^{1/r'} & {\rm if} \ \alpha=1+d/r'.
\end{array} \right.
\end{eqnarray}
Then
\begin{eqnarray} \label{inverse.thm.pf.eq4}
N_2  & \hskip-0.08in \le  & \hskip-0.08in C  \big( \|A^{-1}\|_{{\mathcal B}(\ell^2)} \|A\|_{{\mathcal B}_{r, \alpha}}\big)^{1/\min(\alpha-d/r',1)}
\nonumber\\
& \hskip-0.08in  & \hskip-0.08in \times
\left\{
\begin{array}{ll} 1 & {\rm if} \ \alpha\ne 1+d/r'\\
 \big(\ln \big( \|A^{-1}\|_{{\mathcal B}(\ell^2)} \|A\|_{{\mathcal B}_{r, \alpha}}+1\big)\big)^{1/r'}  & {\rm if} \ \alpha= 1+d/r'. 
\end{array}\right.
\end{eqnarray}

Let
$W_{A, N_2}=\sum_{l=1}^\infty (V_{A, N_2})^l$.
By \eqref{inverse.thm.pf.eq1}, \eqref{inverse.thm.pf.eq2} and \eqref{inverse.thm.pf.eq3},
we have
 \begin{equation} \label{inverse.thm.pf.eq5}
 \|(V_{A, N_2})^l\|_{{\mathcal B}_{r, \alpha; N_2}}\le
C  2^{-l}, \ l\ge 1,
\end{equation}
which implies that
\begin{equation} \label{inverse.thm.pf.eq6}
 \|W_{A, N_2}\|_{{\mathcal B}_{r, \alpha; N_2}}\le C.
\end{equation}

For any $ \lambda_m \in V_{N_2}$ and $ c\in \ell^2(G)$, applying
\eqref{inverse.lem.eq2} repeatedly we obtain
\begin{eqnarray}\label{inverse.thm.pf.eq7}
\|\Psi_{\lambda_m}^{4N_2} c\|_{2}
 & \hskip-0.08in \le &  \hskip-0.08in 2\|A^{-1}\|_{{\mathcal B}(\ell^2)} \|\Psi_{\lambda_m}^{N_2} Ac\|_{2}
+\sum_{\lambda_{m'}\in V_{N_2}} V_{A, N_2}(\lambda_m, \lambda_{m'}) \|\Psi_{\lambda_{m'}}^{4N_2} c\|_{2}\nonumber\\
& \hskip-0.08in \le & \hskip-0.08in \cdots \nonumber\\
& \hskip-0.08in \le & \hskip-0.08in 2\|A^{-1}\|_{{\mathcal B}(\ell^2)} \|\Psi_{\lambda_m}^{N_2} Ac\|_{2}
+ \sum_{\lambda_{m'}\in V_{N_2}}
V_{A, N_2}^{k+1}(\lambda_m, \lambda_{m'}) \|\Psi_{\lambda_{m'}}^{4N_2} c\|_{2}
\nonumber \\
& & \hskip-0.08in
+2\|A^{-1}\|_{{\mathcal B}(\ell^2)} \sum_{l=1}^k \sum_{\lambda_{m'}\in V_{N_2}}
V_{A, N_2}^l(\lambda_m, \lambda_{m'}) \|\Psi_{\lambda_{m'}}^{4N_2} Ac\|_{2}, \ \ k\ge 2.
\end{eqnarray}
Using the argument used to prove  the first conclusion in Proposition  
\ref{beurling.prop},
 we have
\begin{eqnarray}\label{inverse.thm.pf.eq8}
& &
\sum_{\lambda_{m'}\in V_N}
V_{A, N_2}^{k+1}(\lambda_m, \lambda_{m'}) \|\Psi_{\lambda_{m'}}^{4N_2} c\|_{2}\nonumber\\
&\hskip-.008in \le & \hskip-.08in\sum_{\lambda_{m'}\in V_{N_2}}
V_{A, N_2}^{k+1}(\lambda_m, \lambda_{m'}) \| c\|_{2}
\le   C \|(V_{A, N_2})^{k+1}\|_{{\mathcal B}_{r, \alpha; N_2}} \|c\|_2.
\end{eqnarray}
Taking limit in \eqref{inverse.thm.pf.eq7}, we obtain from
 \eqref{inverse.thm.pf.eq5}  and \eqref{inverse.thm.pf.eq8}
 that
\begin{eqnarray}\label{inverse.thm.pf.eq9}
\|\Psi_{\lambda_m}^{4N_2} c\|_{2}
 & \hskip-0.08in \le & \hskip-0.08in  2\|A^{-1}\|_{{\mathcal B}(\ell^2)} \|\Psi_{\lambda_m}^{4N_2} Ac\|_{2}\nonumber\\
 & & \hskip-0.08in
+2\|A^{-1}\|_{{\mathcal B}(\ell^2)}  \sum_{\lambda_{m'}\in V_{N_2}}
W_{A, N_2}(\lambda_m, \lambda_{m'}) \|\Psi_{\lambda_{m'}}^{4N_2} Ac\|_{2},
\end{eqnarray}
where $W_{A,N_2}=\big(W_{A, N_2}(\lambda_m, \lambda_{m'})\big)_{\lambda_m, \lambda_{m'}\in V_{N_2}}$.

Write $A^{-1}= (d(\lambda',\lambda))_{\lambda', \lambda\in V}$ and
 set $d_\lambda= (d(\lambda', \lambda))_{\lambda'\in V}, \lambda\in V$.
 Take  $\lambda, \lambda'\in V$ and let $\lambda_m\in V_{N_2}$ be so chosen that
 \begin{equation}\label{inverse.thm.pf.eq10}
\rho(\lambda', \lambda_m)\le 2N_2. \end{equation}
  The existence of such a fusion vertex $\lambda_m$ follows from the covering property
 in Proposition \ref{prop-counting}.
 Applying \eqref{inverse.thm.pf.eq9} with $c$ replaced by $d_\lambda$, we obtain
 \begin{eqnarray}\label{inverse.thm.pf.eq11}
 |d(\lambda', \lambda)| & \le &  \|\Psi_{\lambda_m}^{4N_2} d_\lambda\|_2
 \le 2\|A^{-1}\|_{{\mathcal B}(\ell^2)} \Big|\psi_0\Big(\frac{\rho(\lambda_m, \lambda)}{4N_2}\Big)\Big|\nonumber\\
 & &
 + 2\|A^{-1}\|_{{\mathcal B}(\ell^2)}\sum_{\lambda_{m'}\in V_{N_2}} W_{A, N_2}(\lambda_m, \lambda_{m'}) \Big|\psi_0\Big(\frac{\rho(\lambda_{m'}, \lambda)}{4N_2}\Big)\Big|.
 \end{eqnarray}
Therefore
\begin{eqnarray} \label{inverse.thm.pf.eq12}
\sup_{\lambda', \lambda\in V} |d(\lambda',\lambda)| & \hskip-0.08in \le & \hskip-0.08in
2\|A^{-1}\|_{{\mathcal B}(\ell^2)}\Big(1+
\Big(\sup_{\lambda_m, \lambda_{m'}\in V_{N_2}} |W_{A, N_2}(\lambda_m, \lambda_{m'})| \Big)\nonumber\\
& & \ \ \times
\big(\sup_{\lambda\in V}\sum_{\lambda_{m'}\in V_{N_2}}\chi_{B(\lambda_{m'}, 4N_2)}(\lambda)\big)\Big)\nonumber\\
& \le &  C\|A^{-1}\|_{{\mathcal B}(\ell^2)}, 
\end{eqnarray}
where the last inequality follows from \eqref{inverse.thm.pf.eq6} and Proposition
\ref{prop-counting}.
For $n\ge 12$, it follows from  \eqref{inverse.thm.pf.eq10} and \eqref{inverse.thm.pf.eq11} that
\begin{eqnarray} \label{inverse.thm.pf.eq13}
\sup_{\rho(\lambda', \lambda)\ge n N_2} |d(\lambda,\lambda')|
 &\hskip-0.08in  \le & \hskip-0.08in 2\|A^{-1}\|_{{\mathcal B}(\ell^2)} g_{A, N_2}(n/2)
\sup_{\lambda\in V}\sum_{\lambda_{m'}\in V_{N_2}}\chi_{B(\lambda_{m'}, 4N_2)}(\lambda)\nonumber\\
& \hskip-0.08in \le  & \hskip-0.08in  C \|A^{-1}\|_{{\mathcal B}(\ell^2)} g_{A, N_2}(n/2),
\end{eqnarray}
where
$g_{A, N_2}(n)= \sup_{\rho(\lambda_m, \lambda_{m'})\ge Nn} |W_{A,N_2}(\lambda_m, \lambda_{m'})|$.

Observe that
$$\|A^{-1}\|_{{\mathcal B}_{r, \alpha}}
  \le C  N_2^{\alpha+d/r}
    \Big(\sum_{m\ge 0} \Big(\sup_{\rho(\lambda, \lambda')\ge mN_2} |d(\lambda,\lambda')| \Big)^r (m+1)^{\alpha r+d-1}\Big)^{1/r}$$
for $1\le r<\infty$, and
$$\|A^{-1}\|_{{\mathcal B}_{r, \alpha}}
  \le   CN_2^\alpha  \sup_{m\ge 0} \Big(\sup_{\rho(\lambda, \lambda')\ge mN_2} |d(\lambda,\lambda')| \Big) (m+1)^{\alpha}$$
for $r=\infty$.  Combining the above two estimates with
 \eqref{inverse.thm.pf.eq6}, \eqref{inverse.thm.pf.eq12} and \eqref{inverse.thm.pf.eq13},
 we obtain
\begin{equation}\label{inverse.thm.pf.eq14}
\|A^{-1}\|_{{\mathcal B}_{r, \alpha}}
\le C \|A^{-1}\|_{{\mathcal B}(\ell^2)}
  N_2^{\alpha+d/r}.
\end{equation}
Hence the desired estimate
\eqref{inverse.thm.eq1} follows from \eqref{inverse.thm.pf.eq4}
and \eqref{inverse.thm.pf.eq14}.
\end{proof}

\section{Norm-controlled powers}
\label{power.section}

 By \eqref{power.preeq111}, norms of powers $A^n, n\ge 1$, of a matrix  $A\in {\mathcal B}_{r,\alpha}$ with $1\le r\le \infty$ and $\alpha>d(1-1/r)$ are dominated by a subexponential function.
In this  section, we  show that norms of powers $A^n, n\ge 1$, are controlled by a  polynomial when the counting measure is normal.

\begin{thm}\label{power.thm}  Let  $1\le r\le \infty, r'=r/(r-1)$, ${\mathcal G}$ be a connected simple graph with Beurling dimension $d$
and normal  counting measure $\mu$, and let $A\in {\mathcal B}_{r, \alpha}({\mathcal G})$ with $\alpha>d/r'$.
Then there exists an absolute positive constant $C$, independent of matrices $A$ and integers $n\ge 1$,  such that
\begin{eqnarray}\label{power.thm.eq1}
\frac{\|A^n\|_{{\mathcal B}_{r, \alpha}}} { \|A\|_{{\mathcal B}(\ell^2)}^{n}}
 & \hskip-0.08in \le & \hskip-0.08in C     n
\Big(
\frac{ n
\|A\|_{{\mathcal B}_{r, \alpha}}}{
\|A\|_{{\mathcal B}(\ell^2)} } \Big)^{(\alpha+d/r)/\min(\alpha-d/r',1)}\nonumber\\
&\hskip-0.08in  & \hskip-0.08in   \times \left\{\begin{array}{ll} 1 &  {\rm if} \ \alpha\ne d/r'+1\\
\Big(\ln\Big( \frac{ n
\|A\|_{{\mathcal B}_{r, \alpha}}}{
\|A\|_{{\mathcal B}(\ell^2)} }+1\Big)\Big)^{(d+1)/r'} & {\rm if} \ \alpha=d/r'+1\end{array}\right.\end{eqnarray}
hold for all integers $n\ge 1$.
\end{thm}

For matrices $A\in {\mathcal B}_{r, \alpha}$ with $\alpha>d/r'+1$,
we obtain from  Theorem \ref{power.thm} that
\begin{equation}\label{power.cor.eq1}
\|A^n\|_{{\mathcal B}_{r, \alpha}}\le
C \Big(\frac{\|A\|_{{\mathcal B}_{r, \alpha}}}{\|A\|_{{\mathcal B}(\ell^2)} } \Big)^{\alpha+d/r}
   n^{\alpha+d/r+1} \|A\|_{{\mathcal B}(\ell^2)}^{n}, \ n\ge 1.
\end{equation}
As shown in  \eqref{power.cor.eq2} below, the estimate  \eqref{power.cor.eq1}
 on powers of matrices in the Beurling algebra  ${\mathcal B}_{r, \alpha}$  is almost optimal.
Let $\delta$ be the delta function  with $\delta(0)=1$ and $\delta(k)\ne 0$ for all nonzero integers $k$.
 Then for the matrix $A_1=( \delta(i-j-1))_{i,j\in \ZZ}$, we have that
 $(A_1)^n=(a(i-j-n))_{i,j\in \ZZ}, n\ge 1$,  and hence
\begin{eqnarray}\label{power.cor.eq2}
\|(A_1)^n\|_{{\mathcal B}_{r, \alpha}} &\hskip-0.08in  = & \hskip-0.08in \left\{\begin{array}{ll}
\Big(\sum_{k=0}^n (k+1)^{\alpha r}\Big)^{1/r}  & {\rm if} \ 1\le r<\infty\\
(n+1)^\alpha & {\rm if} \ r=\infty\end{array}\right.\nonumber\\
 & \hskip-0.08in \ge &  \hskip-0.08in C
 \Big(\frac{\|A_1\|_{{\mathcal B}_{r, \alpha}}}{\|A_1\|_{{\mathcal B}(\ell^2)} } \Big)^{\alpha+d/r}
   n^{\alpha+d/r} \|A_1\|_{{\mathcal B}(\ell^2)}^{n}, \ n\ge 1,
\end{eqnarray}
 where $C$ is an absolute constant. 

Let  
$\hat a(\xi)=\sum_{k\in \ZZ} a(k)e^{-ik\xi}$ satisfy $|\hat a(\xi)|\le 1$ for all $\xi\in \RR$
and $\sup_{k\in \ZZ} |a(k)| (1+|k|)^\alpha<\infty$ for some $\alpha>1$, and
write
$$(\hat a(\xi))^n= \sum_{k\in \ZZ} a_n(k)e^{-ik\xi}, \ n\ge 1.$$
Then there exists a positive constant $C$ independent of $n\ge 1$ such that
$$
|a_n(k)|\le C n^{\alpha+1} (1+|k|)^{-\alpha}, \ k\in \ZZ
$$
by Theorem \ref{power.thm}.
Therefore
 for any $\epsilon>0$, there exists a positive constant $C_\epsilon$ such that
\begin{equation}\label{ank.estimate}
\sum_{k\in \ZZ} |a_n(k)|\le C n^{1+\epsilon}, \ n\ge 1,
\end{equation}
cf.  \cite{randle14,  randle15, thomee69, thomee65} and references therein for
various estimates.
We remark that the above estimate for the Wiener norm of $(\hat a(\xi))^n, n\ge 1$,
was established in \cite{thomee65}, with
the polynomial exponent $1+\epsilon$  replaced by a smaller exponent  $(1-\mu/\nu)/2$,
when
 $$ \hat a(\xi)=e^{-i\alpha \xi+ i\xi^\mu q(\xi)-\gamma \xi^\nu(1+o(1))}$$
near  the origin
for some real polynomial $q$ with $q(0)\ne 0$.

Let random variables $X_n, n\ge 1$, be a stationary Markov chain  on a  spatially distributed network, which is described by a connected simple graph ${\mathcal G}=(V, E)$.
Then the probabilities  $\Pr(X_{n+1}=\lambda \mid X_{n}=\lambda')$ of going from one vertex $\lambda'$ at time $n$
to another vertex $\lambda$ at time $n+1$ is independent of $n\ge 1$,
$$ \Pr(X_{n+1}=\lambda \mid X_{n}=\lambda')= p(\lambda,\lambda'), \ \lambda, \lambda'\in V \ {\rm and} \ n\ge 1. $$
Define the transition matrix of the above stationary Markov chain by  $P=(p(\lambda,\lambda'))_{\lambda, \lambda'\in V}$.
Then by Theorem \ref{power.thm}, we have the following estimate on the probability $ \Pr(X_{m}=\lambda |X_{n}=\lambda'), m>n\ge 1$,  with the input vertex $\lambda'$ and output vertex $\lambda\in V$.

\begin{cor}\label{markov.cor}
Let   ${\mathcal G}:=(V,E)$ be a connected simple graph with Beurling dimension $d$
and normal  counting measure $\mu$, and let $X_n, n\ge 1$, be a stationary Markov chain  on the graph ${\mathcal G}$
with transition matrix $P \in {\mathcal B}_{\infty, \alpha}$ for some $\alpha>d+1$.
Then  there exists a positive constant $C_\alpha$ such that
 \begin{equation}\label{markov.cor.eq} \Pr(X_{m}=\lambda |X_{n}=\lambda')
\le  \min\Big(C_\alpha (m-n)^{\alpha+1} (\rho(\lambda, \lambda'))^{-\alpha}, 1\Big)
 \end{equation}
 for   all  $\lambda, \lambda'\in V$ and $m>n\ge 1$.
\end{cor}

\smallskip

We finish this section with the proof of Theorem \ref{power.thm}.

\begin{proof}[Proof of Theorem \ref{power.thm}]
Let $A\in {\mathcal B}_{r, \alpha}$, and write
\begin{equation*} 
A^n =\frac{1}{2\pi i}
\int_{|z|=(1+1/n)\|A\|_{  {\mathcal B}(\ell^2)}}
z^n (zI-A)^{-1} dz.\end{equation*}
Then
\begin{equation}  \label{power.thm.pf.eq2}
\|A^n\|_{{\mathcal B}_{r, \alpha}}
\le  C
\|A\|_{{\mathcal B}(\ell^2)}^n
\int_{|z|= (1+1/n)\|A\|_{{\mathcal B}(\ell^2)}}
\| (zI-A)^{-1}\|_{{\mathcal B}_{r, \alpha}} |dz|.
\end{equation}

Observe that
for $|z|=\|A\|_{{\mathcal B}(\ell^2)} (1+1/n)$, we have
\begin{equation}  \label{power.thm.pf.eq3}
\|(zI-A)^{-1}\|_{{\mathcal B}(\ell^2)}\le  |z|^{-1} \sum_{l=0}^\infty |z|^{-l} \| A\|_{{\mathcal B} (\ell^2)}^l\le   n
(\|A\|_{{\mathcal B}(\ell^2)})^{-1},
\end{equation}
and
\begin{equation}  \label{power.thm.pf.eq4}
\|zI-A\|_{{\mathcal B}_{r, \alpha}}\le |z|+\|A\|_{{\mathcal B}_{r, \alpha}}
\le   C \|A\|_{{\mathcal B}_{r, \alpha}},
\end{equation}
where the last inequality holds by \eqref{beurling.prop.eq1+}.
By  \eqref{power.thm.pf.eq3}, \eqref{power.thm.pf.eq4} and Theorem \ref{inverse.thm}, we get
\begin{eqnarray}
\|(zI-A)^{-1}\|_{{\mathcal B}_{r, \alpha}}
 &\hskip-0.08in  \le & \hskip-0.08in  C    n (\|A\|_{{\mathcal B}(\ell^2)} )^{-1}
\Big(
\frac{ n
\|A\|_{{\mathcal B}_{r, \alpha}}}{
\|A\|_{{\mathcal B}(\ell^2)} } \Big)^{(\alpha+d/r)/\min(\alpha-d/r',1)}\nonumber\\
&\hskip-0.08in  & \hskip-0.08in   \times \left\{\begin{array}{ll} 1 &  {\rm if} \ \alpha\ne d/r'+1\\
\Big(\ln\Big( \frac{ n
\|A\|_{{\mathcal B}_{r, \alpha}}}{
\|A\|_{{\mathcal B}(\ell^2)} }+1\Big)\Big)^{(d+1)/r'} & {\rm if} \ \alpha=d/r'+1.\end{array}\right.\end{eqnarray}
This together with \eqref{power.thm.pf.eq2} proves \eqref{power.thm.eq1}.
\end{proof}

\noindent{\bf Acknowledgement}:\  The authors would like to thank Professors Karlheinz Gr\"ochenig, Andreas Klotz  and Jose Luis Romero for their help and suggestion for the improvement of the manuscript.

\end{document}